\documentclass[12pt]{amsart}
\usepackage[utf8]{inputenc}
\usepackage[letterpaper, margin=1in]{geometry}
\usepackage[colorlinks=true, allcolors=blue,linkcolor=blue, citecolor=blue]{hyperref}
\usepackage[capitalise]{cleveref}
\usepackage{amsfonts}
\usepackage{amsmath}
\usepackage[symbol]{footmisc}
\usepackage{amssymb}
\usepackage{graphicx}

\usepackage{verbatim}
\usepackage{cleveref}

\newcommand{\norm}[1]{\left\|#1\right\|}
\usepackage{tikz}
\usepackage[shortlabels]{enumitem}
\usepackage{bbm}
\usepackage{enumitem}
\newcommand{\R}{\mathbb R}
\newlist{Properties}{enumerate}{2}
\setlist[Properties]{label=Property \arabic*., font=\textbf, itemindent=*}
\usepackage{amsthm}

\newtheorem{asu}{Assumption}
\newcounter{subassumption}[asu]

\makeatletter
\renewcommand{\p@subassumption}{\theasu}
\makeatother

\newtheorem{Theorem}{Theorem}[section]
\newtheorem{Definition}[Theorem]{Definition}
\newtheorem{Lemma}[Theorem]{Lemma}
\newtheorem{Proposition}[Theorem]{Proposition}

\newtheorem{Remark}[Theorem]{Remark}
\newtheorem{property}[Theorem]{Property}

\newcommand{\intersect}{\cap}
\newcommand{\union}{\cup}
\newcommand{\h}{\mathcal{H}}
\newcommand{\re}{\mathbb{R}}

\newcommand{\mt}{\mathcal{T}}
\newcommand{\mtd}{\mathcal{T}^\delta}

\newcommand{\mf}{\mathcal{F}}

\newcommand{\cone}{\mathcal{C}}
\newcommand{\e}{\varepsilon}
\newcommand{\ms}{\mathcal{S}}

\newcommand{\lt}{\mathcal{L}^n}
\newcommand{\Ln}{\mathcal{L}^n}
\newcommand{\pa}{\partial}

\newcommand{\ssubset}{\subset\joinrel\subset}

\usepackage[mathscr]{euscript}
\DeclareSymbolFont{rsfs}{U}{rsfs}{m}{n}
\DeclareSymbolFontAlphabet{\mathscrsfs}{rsfs}
\numberwithin{equation}{section}

\title[Local minimizer in $\re^n$ with an $(n+1)$-junction.]{Local minimizers in $\mathbb{R}^n$ of vector Allen-Cahn with an $(n+1)$-junction. }
\author{Abhishek Adimurthi}
\address{Department of Mathematics, Indiana University 
Bloomington, IN 47405, USA.}
\email{abadim@iu.edu, abhishek.adimu@gmail.com}

\begin{document}

\begin{abstract}
For a domain $\Omega$ that is a deformation of a unit ball in $\R^n$, we establish the existence of a sequence of local minimizers for the vector Allen-Cahn energy having $n+1$ wells. This sequence converges in the $L^1$ topology to a partition of $\Omega$ whose skeleton is given by a simplex cone that contains an $(n+1)$-junction point. This is accomplished by proving that the partition is an isolated local minimizer of a weighted perimeter problem arising as the associated $\Gamma$-limit of the sequence of Allen-Cahn functionals. The results established in this article generalize those in \cite{Abhi+Peter} which dealt with the case $n=3$. We also weaken the one crucial assumption from \cite{Abhi+Peter}.

\end{abstract}

\maketitle
\pagestyle{myheadings}
\markright{ }

\section{Introduction.}

Weakening an assumption and generalizing the results in
\cite{Abhi+Peter} to dimensions $n\geq 4$, we establish the existence of an isolated $L^1$-local minimizer for a weighted partitioning problem. The minimizing $(n+1)$-chambered partition exists within a suitably constructed bounded domain $\Omega\subset\R^n$. 
A key feature of this partition is that its boundary is the restriction of $\Omega$ to an asymmetric minimizing cone possessing an $(n+1)$-junction point. Through the now standard machinery of $\Gamma$-convergence, by \cite{KohnPS}, the existence of a local minimizer of the associated $(n+1)$-well vector Allen-Cahn (Modica-Mortola) energy immediately follows, namely a diffuse version of the asymmetric minimizing cone.

The articles \cite{Abhi+Peter,Zeimer} establish the analogue of this result in lower dimensions when $n=2,3$. Motivated by the same articles \cite{Abhi+Peter,Zeimer}, the goal here is to construct a domain $\Omega \subset \re^n$ and a partition $\ms$ of $\Omega$ into $n+1$ regions such that $\ms$ has a junction point. Moreover, the constructed partition's interior weighted surface area will be shown to be the smallest among small perturbations in $L^1$ topology. 

 There are two key contributions to this article. The first key contribution of this article is the construction of a suitable domain $\Omega$ to be partitioned by a cone. To do this, first using an $n$-simplex, we will construct a partition of $\re^n$ into $n+1$ regions with the condition that this partition possesses an $(n+1)$-junction point: the unique point where all the $n+1$ regions meet. In the next step, we restrict this partition to the unit ball and deform the surface of the ball appropriately. This deformation is carried out so that, locally near the region where $\pa B(0,1)$ meets the cone, the deformed surface looks like a `convex' dent on the ball. The deformation is explained in a clean and rigorous way later in \cref{Construct}. The difference from \cite{Abhi+Peter,Zeimer} is that the partition and therefore the domain are explicitly defined in lower dimensions when $n=2,3,$ whereas in this article, for dimensions $n\geq 4$, we make use of an $n$-simplex to define a partition and thereby a domain from it. 
 
 The second key difference is in the assumption on the weights of the weighted perimeter problem: in \cite{Abhi+Peter}, we assumed that each weight is strictly greater than half of the others, namely the condition \eqref{asu_old} mentioned below. However, in this article, we assume a slightly weaker version of this estimate i.e., each weight is strictly greater than $(n-2)/n$ of the others: we state this estimate explicitly in Assumption \ref{closeness_of_cij} below. In particular, this weakens the assumption from \cite{Abhi+Peter} when $n=3$.

 Once a suitable domain is constructed that most crucially satisfies Property \ref{property} below, we proceed to establish that the associated partition, say $\mathcal{S}$, is in fact an isolated $L^1$-local minimizer of the weighted perimeter energy. As in \cite{Abhi+Peter}, the proof follows in two steps: first we show that we may restrict the set of competitors to those that are uniformly close to $\mathcal{S}$ on the boundary of $\Omega$; the proof here is similar to that employed in \cite{Abhi+Peter} but now requires an inductive argument. In the next step, we invoke a calibration argument, making use of $n+1$ integrals compared with four integrals in \cite{Abhi+Peter} to conclude that $\mathcal{S}$ is indeed an isolated $L^1$-local minimizer. Finally, with an appeal to \cite{KohnPS},  we conclude the existence of local minimizers of the associated Allen-Cahn energies. 

 Let us make all of this more precise now. For a bounded domain $\Omega$ in $\re^n$, consider the Allen-Cahn energies given by
\begin{equation*}
 E_{\varepsilon}(u,\Omega) := \int_{\Omega} \frac{\varepsilon}{2} |\nabla u|^2 + \frac{1}{\varepsilon} W(u) \,dx, 
\end{equation*}
where $u\in H^1(\Omega;\R^n)$ and $W : \re^n \mapsto [0,\infty)$ is a $C^{1,\alpha}$ potential with $\alpha \in (0,1)$ that vanishes on a set of $n+1$ distinct points $ W_{{\rm zero}}:= \{ p_1, p_2, \dots, p_{n+1}\} \subset \re^n$. We set $\mathcal{L}^n$ to be the $n$-dimensional Lebesgue measure, $\h^{n-1}$ to be the $(n-1)$-dimensional Hausdorff measure and $\pa^*A$ to be the reduced boundary of the set $A$. Then, we define a partition of the domain $\Omega$ as follows:
\begin{Definition}
    For a bounded domain $\Omega$, we say $\mathcal{R} = (R_1,R_2,\dots,R_n,R_{n+1})$ is a partition of $\Omega$ if the following hold:
    \begin{itemize}
        \item Each $R_i \subset \Omega$ is a set of finite perimeter, i.e., $\h^{n-1}( \pa^* R_i \intersect \Omega)$ is finite.
        \item $\mathcal{L}^n\left(\Omega \setminus \left(\union_{i=1}^{n+1}R_i \right)\right) = 0$.
        \item $\mathcal{L}^n(R_i \intersect R_j) = 0$ for any distinct $i,j \in \{1,2,\dots, n+1\}$.
    \end{itemize}
\end{Definition}
The derivation of the $L^1$-$\Gamma$-limit of the sequence $\{E_\varepsilon\}$ is carried out in \cite{Baldo} (see also \cite{PeterVect, FT, Peter88}). The $\Gamma$-limit $E_0(.,\Omega)$ acts on the set of partitions $\mathcal{R} = (R_1,R_2,\dots, R_{n+1})$ of $\Omega$ and is given by 
\begin{equation}\label{formulation_1}
\begin{aligned}
     E_{0} \left(\sum_{i=1}^{n+1}p_i \chi_{R_i},\Omega\right)
     :=  \left\{ \begin{aligned}
         \sum_{1\leq i<j \leq n+1} c_{ij} \h^{n-1} \left( \pa^* R_i \intersect \pa^* R_j  \intersect \Omega \right), \; &\text{ if }\sum_{i=1}^{n+1}p_i \chi_{R_i} \in {\rm BV}(\re^n; W_{{\rm zero}}),\\
         \infty, &\quad \text{otherwise,}
     \end{aligned}
     \right.
\end{aligned}
\end{equation}
 where BV denotes the space of functions having bounded variation. Here the constants $c_{ij}$ are defined by
\begin{equation}\label{defin_c_ij}
     c_{ij} := \inf \left\{ \sqrt{2} \int_0^1 \sqrt{W(\gamma(t))} | \gamma'(t)| dt \left| \begin{aligned}
        &\gamma\in C^1([0,1]; \re^n), \\
        &\gamma(0) = p_i, \gamma(1) = p_j.
    \end{aligned}\right. \right\}, \text{ for } p_i, p_j \in W_{\rm zero}.
\end{equation}

In order to accommodate the situation when $i$ or $j$ equals $n+1$, we sometimes write $c_{ij}$ as $c_{i,j}$ in this article. Furthermore, we will often denote $E_0$ in \eqref{formulation_1} by $E_0(\mathcal{R},\Omega)$, where $\mathcal{R}= (R_1,\dots,R_{n+1})$ is a partition of $\Omega$. Equivalently, one can define $E_0(\mathcal{R},\Omega)$ by 
\begin{equation}\label{formulation_2}
    E_0(\mathcal{R},\Omega)=\sum_{i=1}^{n+1} c_i\h^{n-1}(\pa^*R_i\cap\Omega),
\end{equation}
where the relation between the constants $c_{ij}$ and $c_i$ is given by $c_{ij} = c_i +c_j$, for $i\neq j$. In this article, the weights $c_i$ will be assumed to satisfy Assumption \ref{assumption_c_ij_positive} and  Assumption \ref{closeness_of_cij} below. Before stating the assumptions and their uses, we note that any partition $\ms = (S_1,\dots, S_{n+1})$ of $\Omega$ that locally minimizes $E_0$ must satisfy the following stationarity condition: for $\vec{n}_{i,j}$ denoting the unit normal to $\pa^* S_i \intersect \pa^* S_j$, pointing from $S_i$ to $S_j$, for every distinct triple of indices
$i,j,k \in \{1,2,\dots, n+1\}$, the following should hold true:
\begin{equation}\label{gencij}
c_{ij} \vec{n}_{i,j} + c_{jk} \vec{n}_{j,k} + c_{ki} \vec{n}_{k,i} = 0 .
\end{equation}
The identity \eqref{gencij} can be derived by setting the first variation of $E_0$ locally in some neighborhoods inside $\Omega$ to be zero and this is explained more clearly in Section \ref{Construct} below.

We now mention a criterion that the domain and the partition of it should satisfy. Consider a domain $\Omega$ containing the origin and a partition $\ms = (S_1, S_2, \dots , S_{n+1})$ of $\Omega$ having the origin as a junction point. We say that the pair $(\Omega,\ms)$ is {\it admissible} for Theorem \ref{isolated_local_minimizer} below, if the following property holds:
\begin{property}\label{property}
 There exists a real number $\eta = \eta(\Omega) > 0$ such that, for every distinct pair of indices
$i,j \in \{1,2,\dots, n+1\}$ and for every point
\[
 x \in \partial S_i \cap \partial \Omega \text{ such that }
\operatorname{dist}(x, \pa S_i \intersect \pa S_j \cap \partial \Omega) < \eta ,
\]
one has the following property:
\begin{equation}\label{goodnormals}
\begin{cases}
\vec{n}_{i,j} \cdot \vec{\nu}_{\partial \Omega}(x) > 0, & \text{if } x \notin \pa S_i \intersect \pa S_j \cap \partial \Omega, \\[4pt]
\vec{n}_{i,j} \cdot \vec{\nu}_{\partial \Omega}(x) = 0, & \text{if } x \in \pa S_i \intersect \pa S_j \cap \partial \Omega,
\end{cases}
\end{equation}
where $\vec{n}_{i,j}$ denotes the unit normal to the hyperplane $\pa S_i \intersect \pa S_j \intersect\Omega$ pointing from the region $S_i$ toward $S_j$ and $\vec{\nu}_{\pa \Omega}(x)$ is the unit outward normal to $\Omega$ at $x \in \pa \Omega$.
\end{property}
We will now mention the two assumptions on $c_{ij}$ and their uses, one at a time:
\begin{asu}\label{assumption_c_ij_positive} The quadratic form $F : \R^n \mapsto \R$ given by
\begin{equation}\label{quadratic_form}
    F(x_2,x_3,\dots, x_n, x_{n+1}) := \frac{1}{2}\sum_{i,k=2}^{n+1} \left( c_{1i}^2 + c_{1k}^2 - c_{ik}^2\right) x_ix_k
\end{equation}
is positive definite.
\end{asu}
Assumption \ref{assumption_c_ij_positive} will be used to describe a {\it simplex cone} that will be obtained from an $n$-simplex and that satisfies \eqref{gencij}. By \cite[Theorem 1]{Schoenberg_Annals}, one sees that Assumption \ref{assumption_c_ij_positive} is necessary and sufficient for the existence of an $n$-simplex with vertices $V_1, V_2, \dots, V_{n+1} \in \mathbb{R}^n$ such that the length of the edge connecting $V_i$ and $V_j$ is $c_{ij}$. Then, for all distinct pairs $i,j \in \{1,2,\dots,n+1\}$, a unit normal vector to the hyperplane separating phase $i$ from phase $j$ can be defined by
\begin{equation}\label{defn_of_n_ij}
\vec{n}_{i,j} := \frac{1}{c_{ij}}\; (V_j - V_i).
\end{equation}
 We note here that \eqref{defn_of_n_ij} implies that the {\it simplex cone} that we obtain will satisfy \eqref{gencij}. By construction, for each distinct pair $i,j$, there exists an $(n-1)$-dimensional hyperplane containing the circumcenter of the $n$-simplex and having unit normal $\vec{n}_{ij}$.
 There are $\binom{n+1}{2}$ such $(n-1)$-dimensional hyperplanes in total. When these hyperplanes are considered together, some of them meet some other hyperplanes. So, in each of these hyperplanes, pick the largest subset of the hyperplane that contains the circumcenter, meets the line segment $\overline{V_i V_j}$ orthogonally and the manifold boundary of this part of the hyperplane touches the other hyperplanes. The collection of these $(n-1)$-dimensional hyperplanes has a unique $(n+1)$-junction and we will call this collection of all of the $(n-1)$-dimensional hyperplanes together that partitions the whole of $\re^n$ into $n+1$ regions, a {\it simplex cone}. We will denote this {\it simplex cone} by $\tilde{\cone}$ in this article.  
 
 \begin{Remark}
     By the calibration argument below {\rm(}also see \cite{Morgan}{\rm)}, one sees that the partition $\tilde{\ms}$ of $\re^n$ minimizes the weighted perimeter $E_0$ locally in the sense of De Giorgi. Or equivalently, the infinite asymmetric simplex cone is a minimizing cone in the sense of De-Giorgi. That is, for any compact set $K \subset \re^n$, $E_0(\tilde{\ms},K)$ is the smallest among competitors that agree with $\tilde{\ms}$ outside $K$. However, for the content of this article, we do not assume that the boundary values are fixed.
  \end{Remark}
  
  Upon choosing the co-ordinates appropriately, one can assume that the $(n+1)$-junction is the origin in $\re^n$. For examples of simplex cones, in $\re^2$, $\tilde{\cone}$ is a triad - see \cref{triad} and in $\re^3$, the corresponding simplex cone is a tetrahedral cone - see \cref{Tetrahedral_cone}. We will denote the partition and therefore the regions of $\re^n$ partitioned by this simplex cone $\tilde{\cone}$ by $\tilde{\ms} = (\tilde{S}_1,\tilde{S}_2,\dots,\tilde{S}_{n+1})$. In \cref{Construct}, we will use this cone to describe a domain $\Omega$ that is admissible for \cref{isolated_local_minimizer} below i.e., that satisfies Property \ref{property}. Later on, we will use the ideas in \cite{Abhi+Peter} to conclude the proof of \cref{isolated_local_minimizer}.

  We will make another assumption on the constants $c_{ij}$, namely

  \begin{asu}\label{closeness_of_cij}
      The constants $c_i$ defined in \eqref{formulation_2} are assumed to satisfy:
      \begin{equation}\label{cijpospos}
        \text{\rm For all distinct }i,j \; \text{\rm in } \{ 1,2,\dots, n+1\},\; c_{i}> \left(\frac{n-2}{n}\right)c_{j}.
      \end{equation}
  \end{asu}
The author believes that Assumption \ref{closeness_of_cij} could be removed entirely but could not find a way to eliminate this assumption. Moreover, the condition \eqref{cijpospos} will later be used in the corralling step i.e., in the proof of \ref{3.2.2} of Theorem \ref{lemma1}.
\begin{figure}
    \centering
    \includegraphics[width=0.9\linewidth]{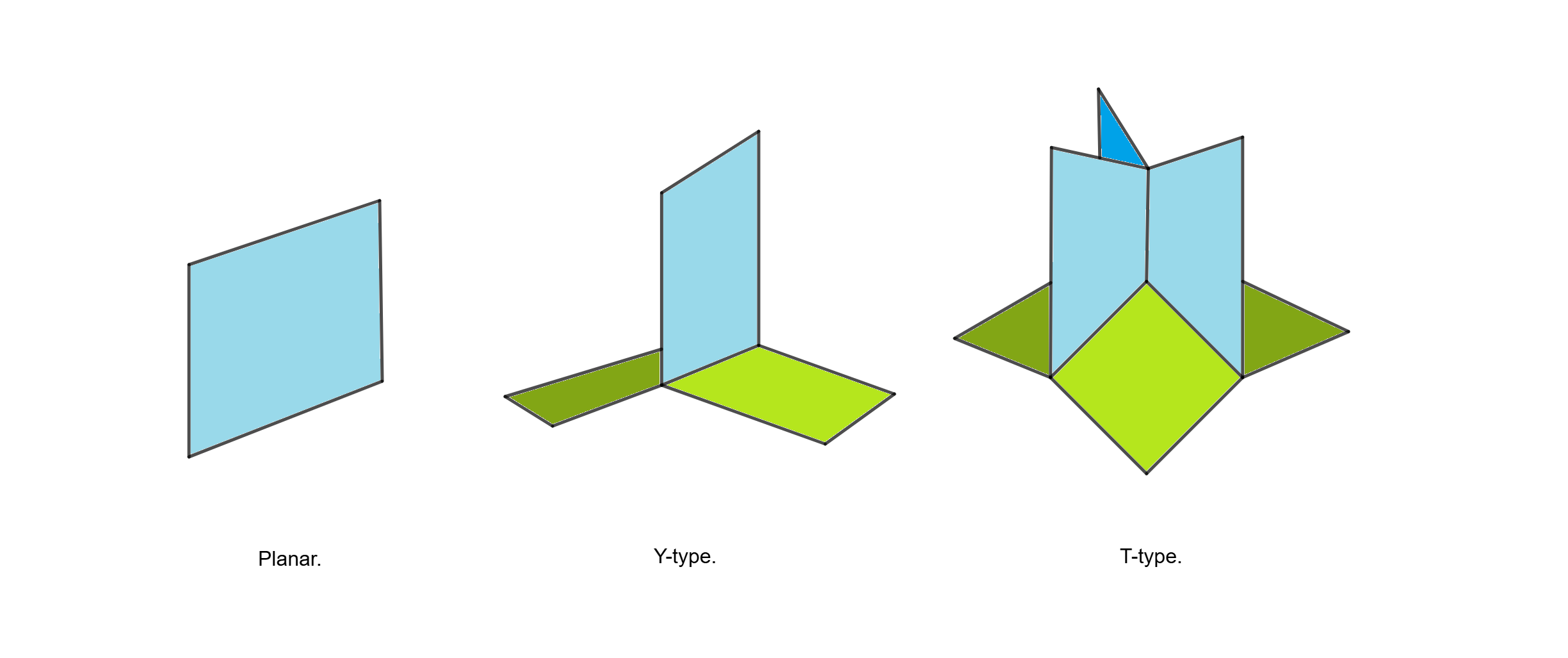}
    \caption{Classification of minimizing symmetric tangent cones in $\re^3$ by Taylor \cite{Taylor}.}
    \label{Taylor}
\end{figure}


 \begin{Remark}
    For some $c>0$, if $c_{i} = c$ for all $1 \leq i \leq n+1$, then the right-hand side {\rm (R.H.S.)} of \eqref{quadratic_form} takes the form 
    \[
    \frac{1}{2}\sum_{i,k=2}^{n+1}\left( c_{1i}^2 + c_{1k}^2 - c_{ik}^2\right) x_ix_k = \frac{4c^2}{2} \left( \sum_{i=2}^{n+1} x_i \right)^2,
    \]
    which is positive for any $(x_2,x_3 ,\dots, x_{n+1}) \in \R^n$. Also, in this case when all the constants $c_i$ are equal to $c$, the inequality in \eqref{cijpospos} turns out to be
    \[
    1 > \frac{n-2}{n} \iff n > n-2,
     \]
     which is a true statement for all $n\in \mathbb{N}$. Therefore, the condition \eqref{cijpospos} implies that the constants $c_i$ are not too far from each other. We note here that the constants $c_{ij}$ defined from \eqref{defin_c_ij} need not always satisfy Assumption \ref{assumption_c_ij_positive}. Moreover, in \cite{Abhi+Peter}, we assumed that the constants $c_{i}$ satisfy the relation: 
    \begin{equation}\label{asu_old}
    \text{ \rm In the case $n=3$, for any distinct $i,j \in \{1,2,3,4\}$, one has }2c_i > c_j.
    \end{equation}
    In this article, Assumption \ref{closeness_of_cij} is slightly more general than the one in \eqref{asu_old} assumed in \cite{Abhi+Peter} as \eqref{cijpospos} assumes that $3c_i>c_j$ for all distinct $i,j$ in the case when $n=3$.

    Additionally, we claim that Assumption \ref{closeness_of_cij} implies Assumption \ref{assumption_c_ij_positive} and the proof of this claim is given in Proposition \ref{propn_asu} below. Although this claim makes Assumption \ref{assumption_c_ij_positive} redundant, we nevertheless stated it above to explain the existence of an  asymmetric simplex cone.
\end{Remark}

In \cite{Taylor}, Taylor showed that the symmetric tangent cones minimizing the perimeter in $\re^3$ are either a plane or three planes meeting along a line or a tetrahedral cone; see \cref{Taylor}. Here, a symmetric cone refers to the setting where if two planes meet along a line, then the angle between them is $120^{\circ}$. An analogous classification is not yet known in the asymmetric setting in $\re^n$ for $n\geq 3$. A possible approach is to use the results from \cite{Abhi+Peter} and this article. One can also use the results in this article to study the classification of the blowups and blowdowns of local minimizers in higher dimensions. The main result of this article is the following theorem:
\begin{Theorem}\label{isolated_local_minimizer}
    Let $c_{ij}$ be constants that satisfy Assumptions \ref{assumption_c_ij_positive} and \ref{closeness_of_cij}. 
    Then, there exists a domain $\Omega \subset \R^n$ and a partition $\ms = (S_1,S_2,\dots,S_n, S_{n+1})$ of $\Omega$, having an $(n+1)$-junction such that the pair $(\Omega, \ms)$ satisfies Property \ref{property} and that $\ms$ is an isolated $L^1$-local minimizer of $E_0(.,\Omega)$. More precisely, there exists $\delta>0$ such that for any partition $\mt$ of $\Omega$, one has
\begin{equation*}
E_0(\mt,\Omega) > E_0(\ms,\Omega)\quad\text{provided}\quad 0< \Ln \left( \mt \Delta \ms\right) \leq \delta.   
\end{equation*}
\end{Theorem}
Here, for $\mathcal{R}:= (R_1,R_2,\dots, R_{n+1})$ and $\mathcal{R}':= (R_1',R_2',\dots, R_{n+1}')$ partitions of $\Omega$, we denote the symmetric difference of these partitions by
\[
\mathcal{R}\Delta \mathcal{R}' := \bigcup_{i=1}^{n+1} \Big[(R_i \setminus R_i') \union (R_i' \setminus R_i)\Big].
\]

\begin{figure}
    \centering
    \includegraphics[width=0.2\linewidth]{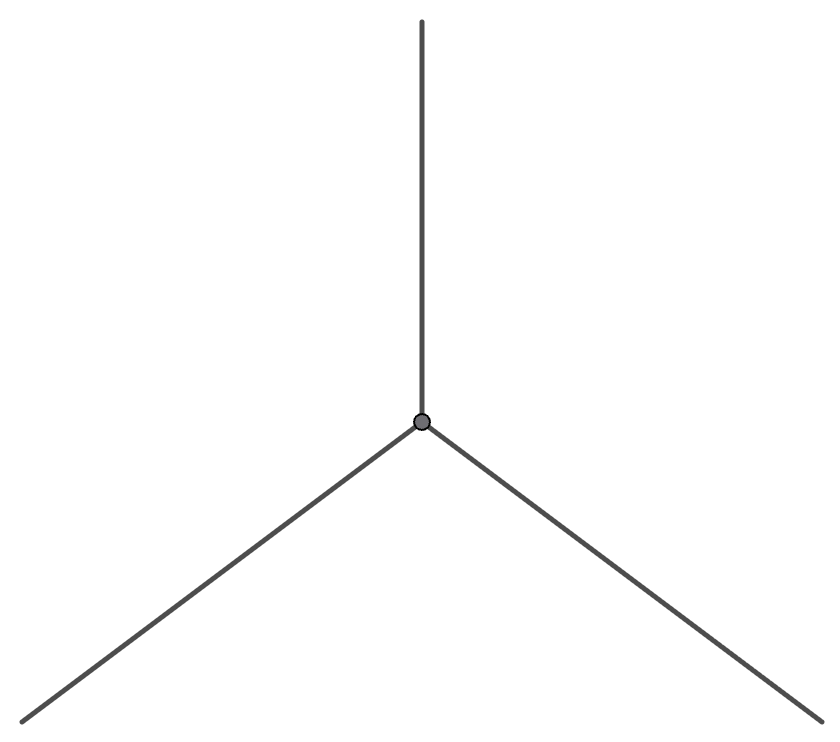}
    \caption{A simplex cone in $\re^2$ is a triad. A triad has a triple junction and partitions the whole of $\re^2$ into 3 regions. When $c_{ij}$ are all equal, then the angle between any two lines is $120^{\circ}$. }
    \label{triad}
\end{figure}

\begin{figure}
    \centering
    \includegraphics[width=0.5\linewidth]{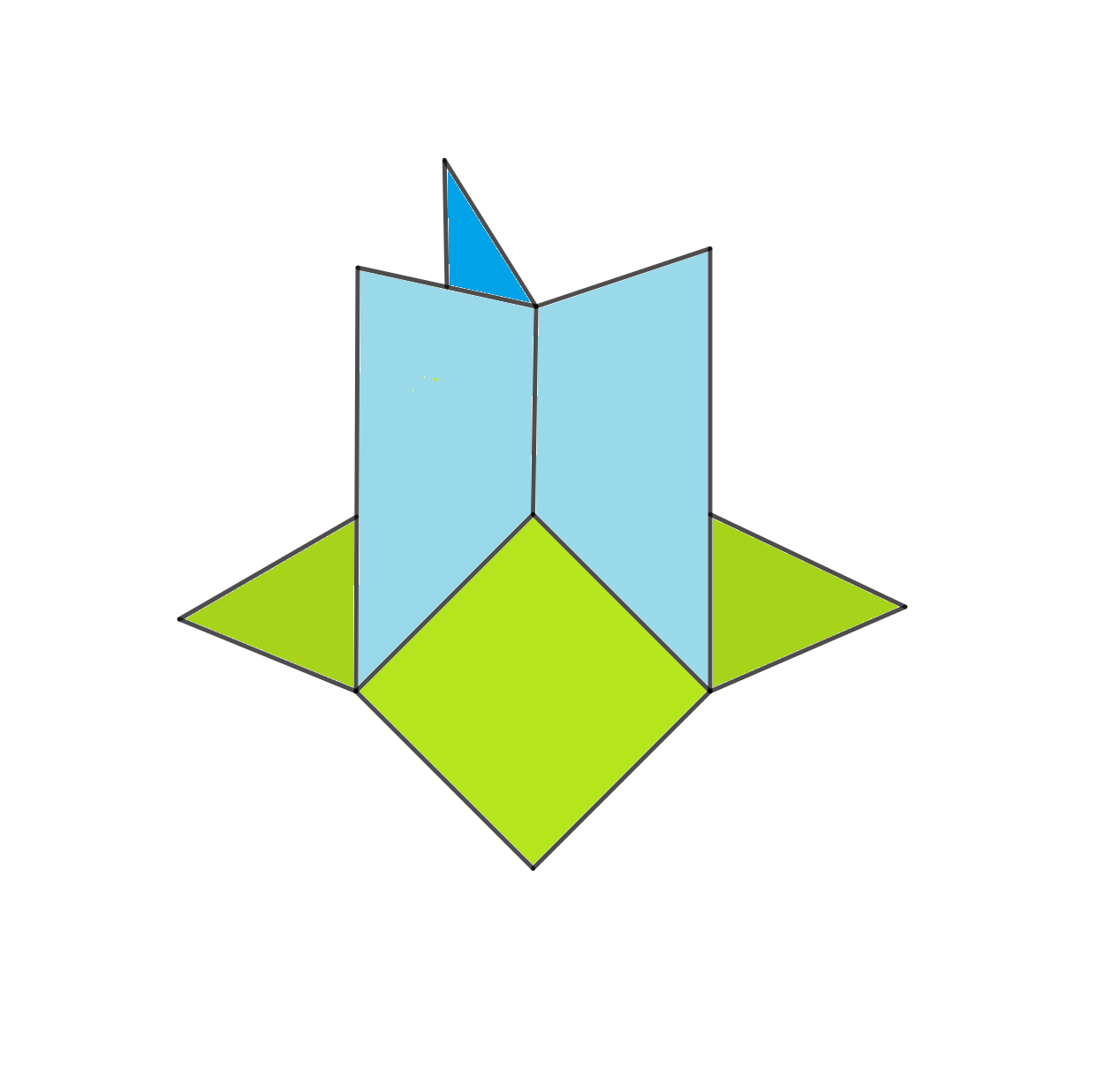}
    \caption{A simplex cone in $\re^3$ is a tetrahedral cone. A tetrahedral cone has a quadruple junction and partitions the whole of $\re^3$ into 4 regions. When $c_{ij}$ are all equal, then the angle between any two planes that meet along a line is $120^{\circ}$.}
    \label{Tetrahedral_cone}
\end{figure}

In lower dimensions, when $n=2$, the results in \cite{Zeimer} establish the existence of an isolated $L^1$-local minimizer where the corresponding simplex cone is a triad. When $n=3$, in \cite{Abhi+Peter} we show the existence of an isolated $L^1$-local minimizer for a partition whose simplex cone is a tetrahedral cone. As an immediate consequence of \cite[Theorem 4.1]{KohnPS}, Theorem \ref{isolated_local_minimizer} implies the following theorem:
\begin{Theorem}\label{pde}
    There exists a domain $\Omega\subset\R^n$ and a partition $\ms:=(S_1,\dots, S_{n+1})$ of $\Omega$ such that each $S_i$ has finite perimeter and that $\ms$  possesses an $(n+1)$-junction point. Further, set 
     \[u_0(x) := \sum_{i=1}^{n+1}p_i \chi_{S_i}(x) \;, \; \text{ for } x \in \Omega.\]
     Then, for all $\e$ sufficiently small, there exists an $L^1$-local minimizer $u_\e$ of the energy $E_\e(.,\Omega)$, satisfying the property that as $\e$ goes to 0, $u_\e\to u_0$ in the topology of $L^1(\Omega)$.
\end{Theorem}

\noindent\underline{\bf Organization of the article:} In Section \ref{Construct}, we give an explicit example of a domain and a partition that is admissible for Theorem \ref{isolated_local_minimizer}. In \cref{require}, we mention a couple of results that we will need for the next section and also prove the claim stating that Assumption \ref{closeness_of_cij} implies Assumption \ref{assumption_c_ij_positive}. In the following section, we present a boundary version of a corralling result, which we then combine with a calibration argument to complete the proof of the Theorem \ref{isolated_local_minimizer}.

\section{Construction of an admissible domain $\Omega$ and the partition $\ms$ of $\Omega$.}\label{Construct}

In this section, we give an example of a domain $\Omega$ and a partition $\ms$ satisfying Property \ref{property}. At the end of this section, we will show that the geometric objects in fact satisfy \eqref{goodnormals}.  

To begin, let $\tilde{\cone}$ be the {\it simplex cone} and $\tilde{\ms}$ be the partition of $\re^n$ defined in the introduction. Upon restricting, one sees that $\tilde{\cone}$ partitions $B(0,1)$ into $n+1$ chambers with each chamber being $\tilde{S_i} \intersect B(0,1)$. The topological boundary of any of these chambers $\pa(\tilde{S_i} \intersect B(0,1))$ consists of $n$ many $(n-1)$-dimensional hyperplanes and a portion of $\partial B(0,1)$. Further, $\tilde{\cone} \cap \partial B(0,1)$ consists of $0$-dimensional points, $1$-dimensional segments of a curve, $2$-dimensional pieces of a surface, and so on, up to parts of $(n-2)$-dimensional hypersurfaces. Away from the origin, consider a small open ball in the interior of $\Omega$ near these $0$-dimensional points that lie on $\tilde{\cone} \intersect \pa B(0,1)$. Inside this ball, one can consider a smaller open ball - say of radius $\tilde{\eta}$ so that this ball $B_{\tilde{\eta}}$ interacts with the cone $\tilde{\cone}$ in the following way: for all $0<r<\tilde{\eta}$, the open ball $B_r\subset B_{\tilde{\eta}}$ touches the same set of hyperplanes among the $\binom{n+1}{2}$ many of them in $\tilde{\cone}$.

In $B_{\tilde{\eta}}$, the cone $\tilde{\cone}$ resembles a configuration equivalent to a structure where three $(n-1)$-dimensional hyperplanes meet along a line segment. Considering variations of the partition in this neighborhood $B_{\tilde{\eta}}$ and imposing the condition that the first variation of $E_0$ in $B_{\tilde{\eta}}$ should vanish, lead to the balance identity \eqref{gencij}.  

The next step is to introduce \emph{convex dents} on the ball at the locations where the cone $\tilde{\cone}$ meets the boundary of the ball, in such a way as to satisfy \eqref{goodnormals}. To provide intuition, we first look at examples from the lower-dimensional cases: in $\mathbb{R}^2$, a domain shown in \cref{2D_Omega} satisfies \eqref{goodnormals}, and in $\mathbb{R}^3$, a domain that looks like \cref{3D_Omega} satisfies \eqref{goodnormals}. Motivated by these figures, we describe a \emph{denting algorithm} to construct $\partial \Omega$ so that \eqref{goodnormals} is satisfied. The first step of the algorithm is to define the \emph{dents} at all the relevant locations on $\pa B(0,1)$ and later, we join these dents appropriately. For example, in $\mathbb{R}^2$, a dent is illustrated in \cref{2d-dent}, while in $\mathbb{R}^3$, appropriately joined dents are shown in \cref{POI}.  At this point, there will be gaps in between these connected dents. These gaps are then connected by attaching a piece of a smooth $(n-1)$-dimensional surface in a Lipschitz manner. This topologically constructed closed $(n-1)$-dimensional surface is what we call $\pa \Omega$ in this article and the region inside it is referred to as $\Omega$. By putting the dents on the surface of the unit ball sufficiently far away from the origin, one can also impose the condition that $B(0,1/4)$ is compactly contained in the region bounded by this $(n-1)$-dimensional surface. Let $S_{l_k,k}$ be the geometric object where $l_k$ many $k$ dimensional surfaces  each having finite $\h^k$ measure, are joined along their manifold boundaries in a Lipschitz manner. Then $\h^{n-1}(S_{k,l_k}) =0$ for any $0\leq k\leq n-2$ and $\max_{0\leq k \leq n-2}l_k < \infty$. Since \eqref{goodnormals} holding true on $\pa \Omega$ is the same as \eqref{goodnormals} holding true on $\pa\Omega$ outside sets having $\h^{n-1}$ measure zero, one sees that  joining the surfaces in a Lipschitz manner is sufficient.  

For a diffeomorphism $f: A \mapsto f(A)$, we say that $f(A)$ is a \emph{diffeomorphic image} of $A$ if $f(A)$ is smooth and \emph{maximal} in the sense that the image cannot be smoothly extended further. For example, consider the graph of $y = |x|$ for $x \in [-1,1]$. This graph can be split into two diffeomorphic images of a unit line segment:  
\[
l_{\rm left}:= \{y = -x \mid -1 \le x \le 0\} \quad \text{and} \quad l_{\rm right}:=\quad \{y = x \mid 0 \le x \le 1\}.
\]  
Here, the corresponding diffeomorphisms are $f(x) = -x$ on $l_{\rm left}$ and $f(x) = +x$ on $l_{\rm right}$. The origin is a point where the graph is not differentiable, so the line segments $l_{\rm left}$ and $l_{\rm right}$ cannot be smoothly extended through zero on the graph of $y = |x|$. This demonstrates that the two diffeomorphic images of the line segment are indeed maximal.

\begin{figure}
    \centering
    \includegraphics[width=0.3\linewidth]{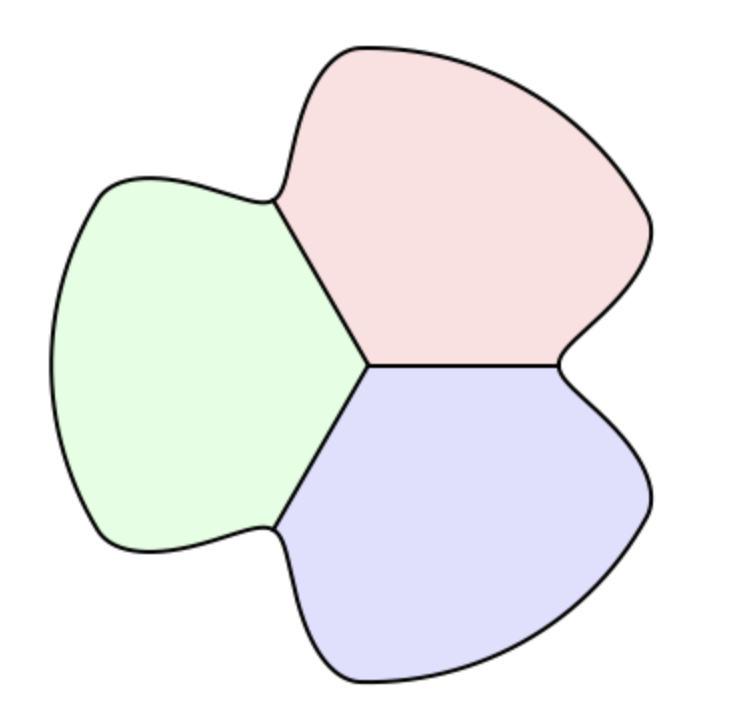}
    \caption{An admissible domain $\Omega \subset \re^2$ that has a triple junction and satisfies \eqref{goodnormals}. }
    \label{2D_Omega}
\end{figure}

\begin{figure}
    \centering
    \includegraphics[width=0.4\linewidth]{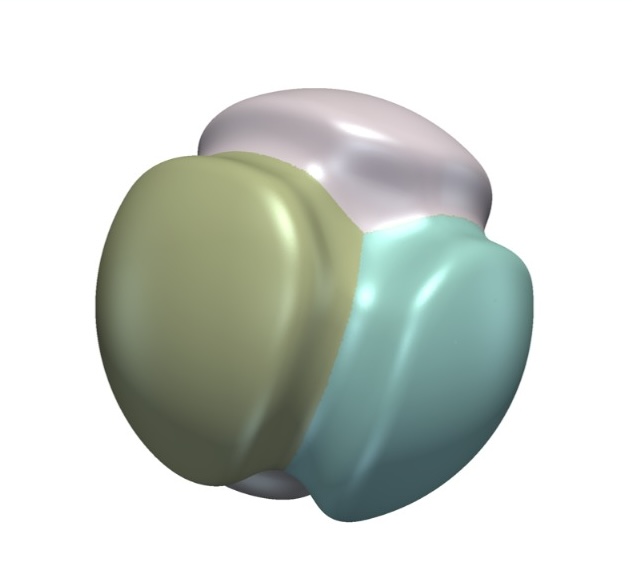}
    \caption{An admissible domain $\Omega \subset \re^3$  satisfying \eqref{goodnormals} that is partitioned by a tetrahedral cone and thereby has a quadruple junction.}
    \label{3D_Omega}
\end{figure}

\begin{figure}
    \centering
    \includegraphics[width=0.3\linewidth]{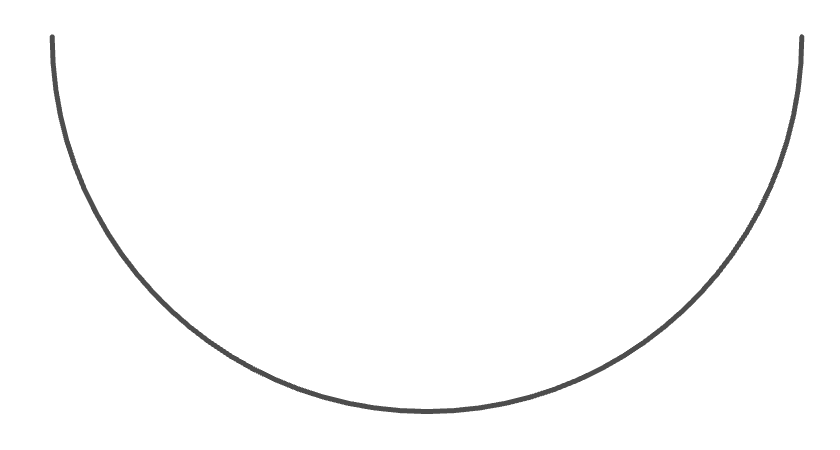}
    \caption{The above is an example of a `dent' in the case of $\re^2$. Attaching three of these `dents' to a triad in \cref{triad} orthogonally and completing the surface, one obtains $\pa \Omega$ in \cref{2D_Omega}.}
    \label{2d-dent}
\end{figure}

\begin{figure}
    \centering
    \includegraphics[width=0.3\linewidth]{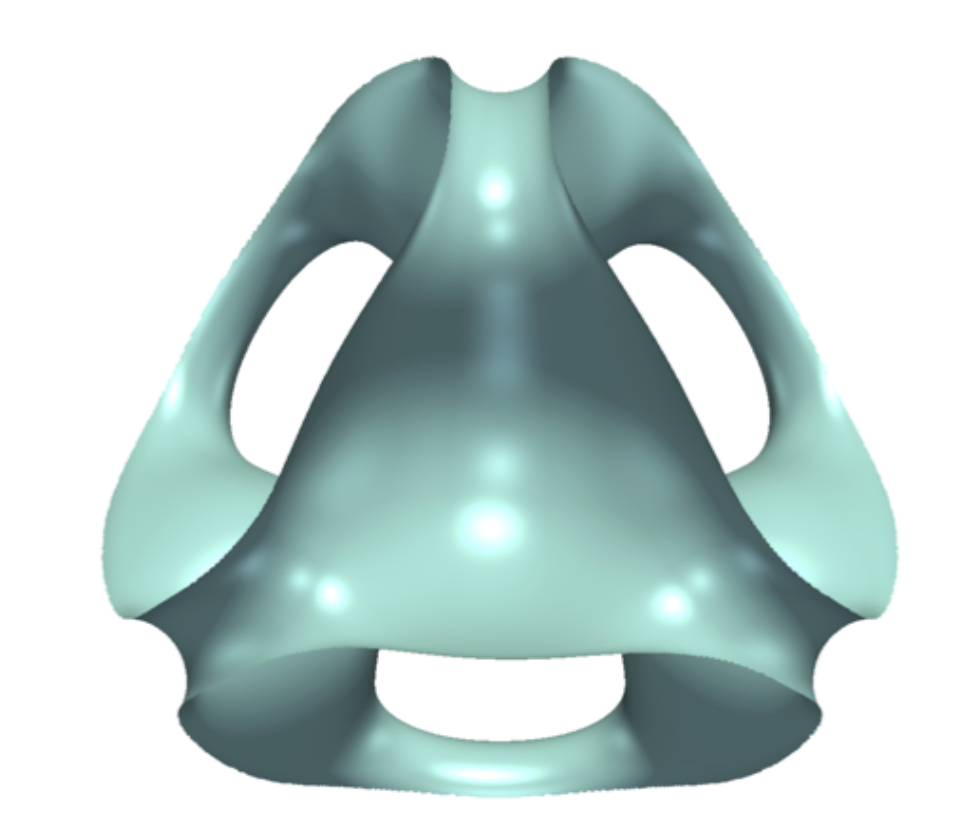}
    \caption{The above troughs and valleys are the `dents' in the case of $\re^3$. Attaching these `dents' to a tetrahedral cone in \cref{Tetrahedral_cone} orthogonally and completing the surface, one gets the $\pa \Omega$ in \cref{3D_Omega}.}
    \label{POI}
\end{figure}

To describe what we mean by dents, we note that $\tilde{\cone} \cap \partial B(0,1)$ consists of diffeomorphic images of $k$-disks for $0 \le k \le n-2$. Here, for $1 \le k \le n-2$, a $k$-disk is defined as
\[
D^k := \left\{ (x_1, x_2, \dots, x_k) \times (0,0,\dots,0) \in \mathbb{R}^k \times \mathbb{R}^{n-k} \ \Big| \ \sum_{i=1}^k x_i^2 \le 1 \right\},
\]
and $D^0$ is taken to be a single point. Fix $k \in \{0,1,\dots, n-2\}$ and suppose that there are $l(k)$ many diffeomorphic images of $D^k$ on $\tilde{\cone} \cap \partial B(0,1)$. Pick a $j \in \{1,2,\dots, l(k)\}$ and by definition, for this choice of $j$ and $k$, there exists a diffeomorphism
\[
h_{j,k}: D^k \mapsto \tilde{\cone} \cap \partial B(0,1)
\]
such that $h_{j,k}(D^k)$ cannot be smoothly extended further on $\partial B(0,1)$. Further, let $P$ be one of the $\binom{n+1}{2}$ many $(n-1)$-dimensional hyperplanes restricted to the closed unit ball. Necessarily, this plane orthogonally meets $\pa B(0,1)$ along $h_{j,k}(D^k)$. For intuition in lower dimensions, in $\re^2$, one can think of $h_{j,k}(D^k)$ as one of the three points (when k = 0) as seen in \cref{2D_Omega} where the triad meets $\partial \Omega$. In $\mathbb{R}^3$, the corresponding $h_{j,k}(D^k)$ may be one of the four points (when $k=0$) or one of the six arcs (when $k=1$) where the tetrahedral cone meets $\partial \Omega$, as illustrated in \cref{3D_Omega}.

Assuming the plane $P$ is  $\{x_{n-1}=0\}$, we will define some model geometries below and later, we will mention some admissible actions that will be used to describe the dents in a neighborhood of the set
$h_{j,k}(D^k) \subset \partial B(0,1)$. Once the admissible actions are defined, we will let $P$ be the plane in $\tilde{\cone}\intersect B(0,1)$ as before and continue the analysis thereafter. To do so, for $k \geq 1$, we
consider a `trough' defined by
\begin{equation}\label{model_geometories}
D^k \times \underline{S}^{\,n-k-1}
:=\left\{(x_1,\dots,x_n)\in\mathbb{R}^n \,\middle|\,
\sum_{i=1}^k x_i^2 \leq 1,\;
\sum_{j=k+1}^{n} x_j^2 = 1,\;
x_n \leq 0
\right\}.
\end{equation}
The geometric objects in \eqref{model_geometories} are also sometimes called model geometries in this article and the notation $\underline{S}^{m}$ denotes the lower hemispherical
subset of the $m$--sphere $S^{m}$. In the particular case when $n=3$ and $k=1$, the trough attached to a plane is depicted in \cref{Plane+trough}, while a lateral view of this configuration is shown in \cref{lateral}. When $k=0$, the corresponding `trough' reduces to a half--sphere, which
we define by
\[
D^0 \times \underline{S}^{\,n-1}
:=\left\{(x_1,\dots,x_n)\in\mathbb{R}^n \,\middle|\,
\sum_{j=1}^{n} x_j^2 = 1,\;
x_n \leq 0
\right\}.
\]

\begin{figure}
    \centering
    \includegraphics[width=0.6\linewidth]{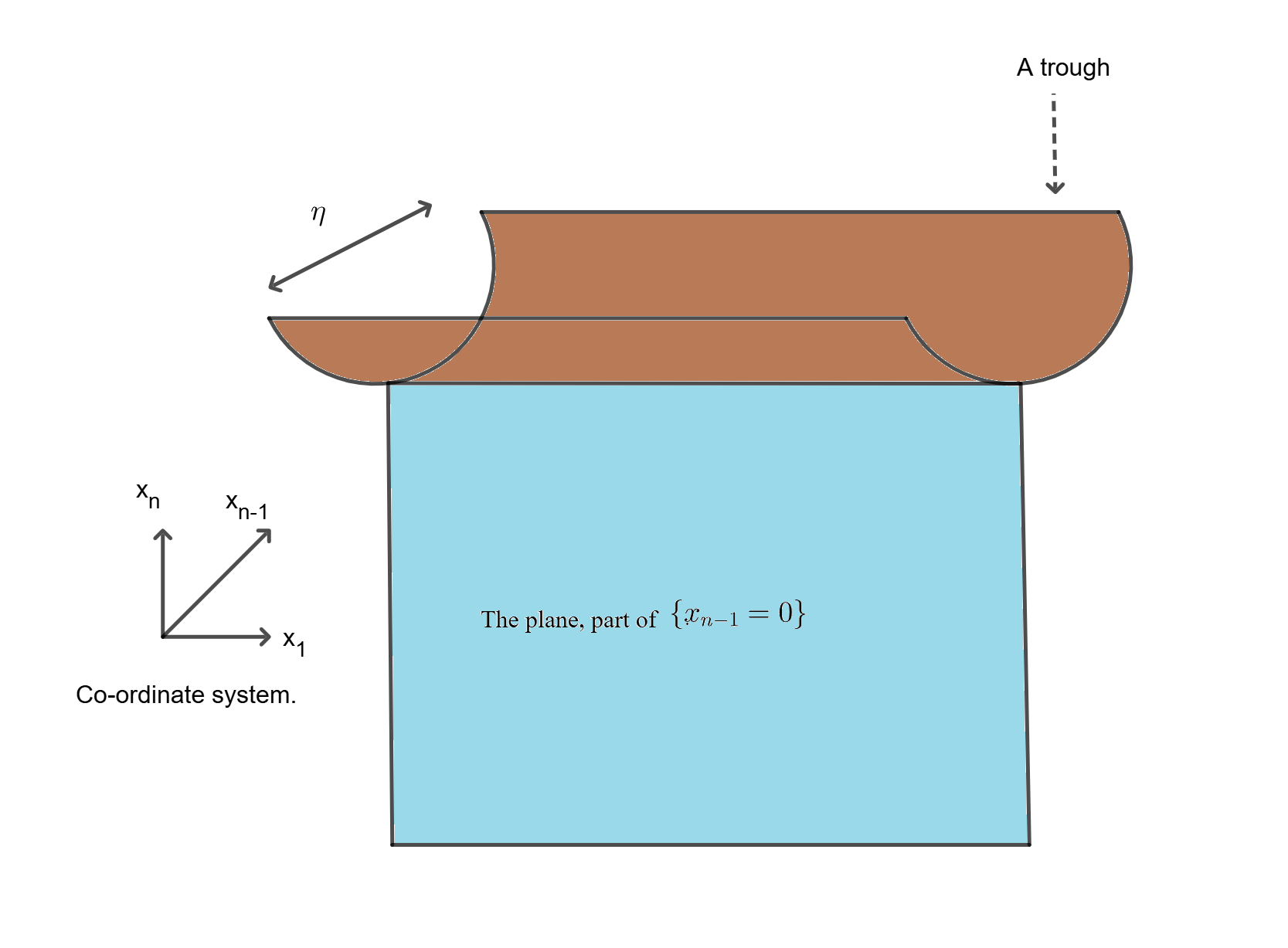}
    \caption{ When $n = 3$ and $k=1$, the trough $D^1 \times \underline{S}^{1}$ is placed orthogonally to the plane that is a part of $\{x_2 = 0\}$. Here $\eta$ is the width of the trough.}
    \label{Plane+trough}
\end{figure}

\begin{figure}
    \centering
    \includegraphics[width=0.6\linewidth]{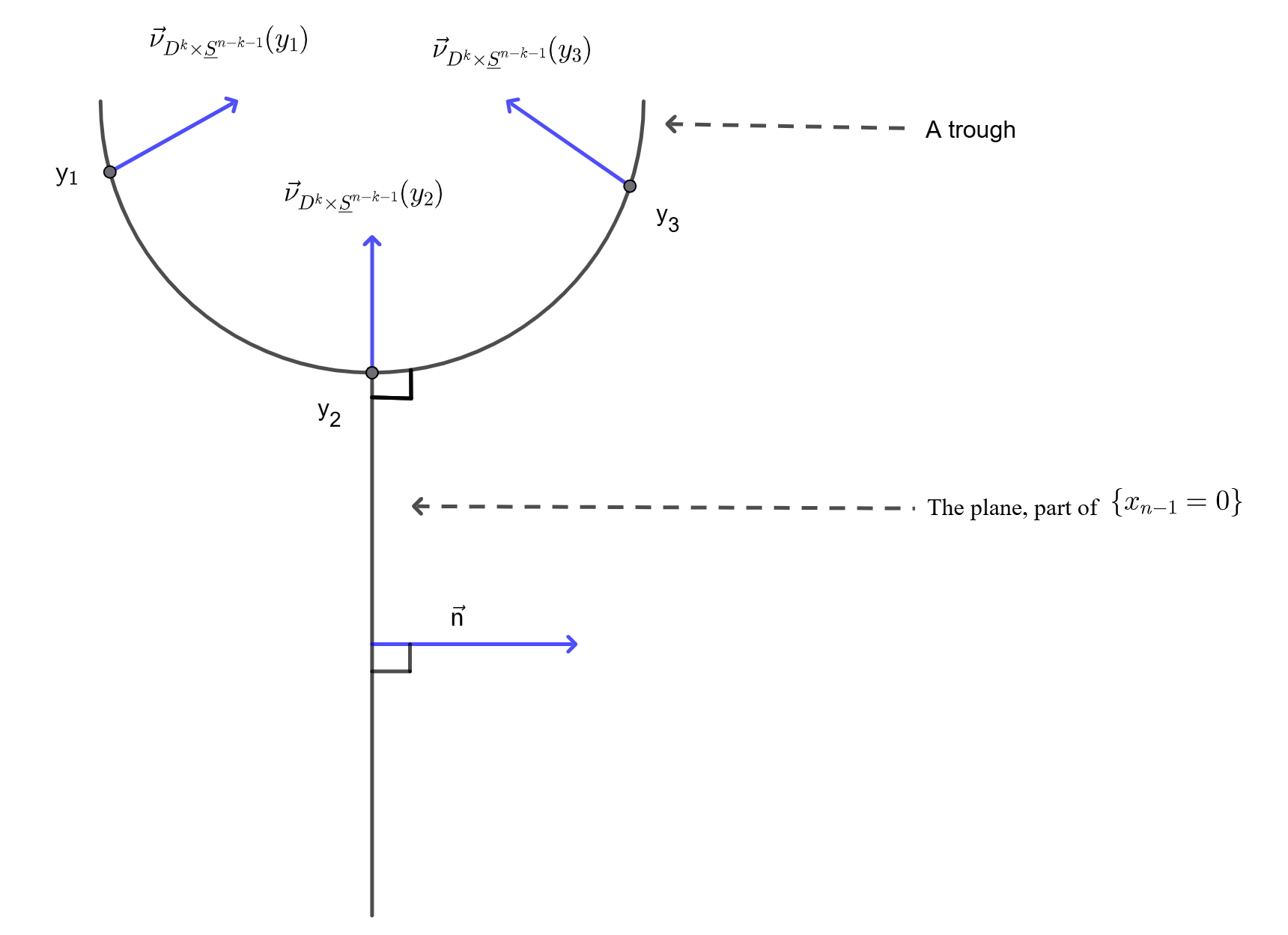}
    \caption{Lateral view of \cref{Plane+trough}. Here $\vec{n}$ is the unit normal to the plane $\{x_{n-1}=0\}$ and $ \vec{\nu}_{D^k \times \underline{S}^{n-k-1}}$ is the unit normal to the trough. The Property \ref{property} or, equivalently \eqref{verify_inequality} is satisfied here: $\vec{n} \cdot \vec{\nu}_{D^k \times \underline{S}^{n-k-1}}(y_1)$ is of one sign, $\vec{n} \cdot \vec{\nu}_{D^k \times \underline{S}^{n-k-1}}(y_3)$ is of the other sign and $\vec{n} \cdot \vec{\nu}_{D^k \times \underline{S}^{n-k-1}}(y_2)=0$. }
    \label{lateral}
\end{figure}
 Lemma \ref{verification} below shows that the set $D^k \times \underline{S}^{\,n-k-1}$ that meets the hyperplane $\{x_{n-1}=0\}$ orthogonally and satisfies \eqref{goodnormals} in Property \ref{property}. Therefore, assuming that \eqref{goodnormals} holds true for all the model geometries in \eqref{model_geometories}, we now describe several operations which we refer to as admissible actions that preserve  \eqref{goodnormals} or equivalently, \eqref{verify_inequality} mentioned below; also see
\cref{Actions}: 
\begin{itemize}
    \item If one translates the set $D^k \times \underline{S}^{\,n-k-1}$
together with the hyperplane $\{x_{n-1}=0\}$, then the translated
configuration continues to satisfy the corresponding inequality
\eqref{verify_inequality}, with the plane $\{x_{n-1}=0\}$ replaced by its
translate and the `trough' replaced by its translated image.
\item One can rotate the entire configuration consisting of
$D^k \times \underline{S}^{\,n-k-1}$ together with the plane $\{x_{n-1}=0\}$ by any
rigid motion of $\mathbb{R}^n$. Such rotations preserve orthogonality
between the trough and the plane, and hence the transformed
configuration again satisfies \eqref{verify_inequality}.
\item One can bend the trough
$D^k \times \underline{S}^{\,n-k-1}$ along the plane $\{x_{n-1}=0\}$ in a
smooth manner and keep the transformed trough orthogonal to
the plane along their intersection. In this case, the corresponding
inequality \eqref{verify_inequality} continues to hold. That is, for a parameter $R_1\geq 0$, one can define the bent trough by
\begin{equation}\label{r_plus_one_stuff}
\left\{
(x_1,\dots,x_n)\in\mathbb{R}^n \left|\sum_{i=1}^k x_i^2 \leq 1, x_n + R_1\sum_{i=1}^kx_i^2 \leq 0,\sum_{j=k+1}^{n-1} x_j^2 + \left(x_n + R_1\sum_{i=1}^kx_i^2\right)^2 = 1.
\right.\right\}.
\end{equation}
Computing the gradient of the identity $\sum_{j=k+1}^{n-1} x_j^2 + \left(x_n + R_1\sum_{i=1}^kx_i^2\right)^2 = 1,$ one sees that the $(n-1)^{th}$ component of the unit normal to the surface described in \eqref{r_plus_one_stuff} is one of $\{- 2x_{n-1}, +2 x_{n-1}\}$ and so, the corresponding conclusion of Lemma \ref{verification} still holds true under this transformation.
\item Instead of the set $D^k \times \underline{S}^{\,n-k-1}$, for
any $R_2>0$ one may consider the stretched trough
\begin{equation}\label{r_stuff}
\left\{(x_1,\dots,x_n)\in\mathbb{R}^n \,\middle|\,
\sum_{i=1}^k x_i^2 \leq R_2,\;
\sum_{j=k+1}^{n} x_j^2 = 1,\;
x_n \leq 0
\right\},
\end{equation}
which corresponds to a stretching in the direction parallel to the plane $\{x_{n-1}=0\}$. This operation also preserves the validity of \eqref{verify_inequality}.
\item\label{shrink} Finally, one can shrink the trough by $R_3$ units:
\begin{equation*}
    \left\{(x_1,\dots,x_n)\in\mathbb{R}^n \,\middle|\,
\sum_{i=1}^k x_i^2 \leq 1,\;
\sum_{j=k+1}^{n} x_j^2 = R_3,\;
x_n \leq 0
\right\}
\end{equation*}
and the corresponding condition \eqref{verify_inequality} is still preserved.
\end{itemize}

\begin{figure}
    \centering
    \includegraphics[width=1.1\linewidth]{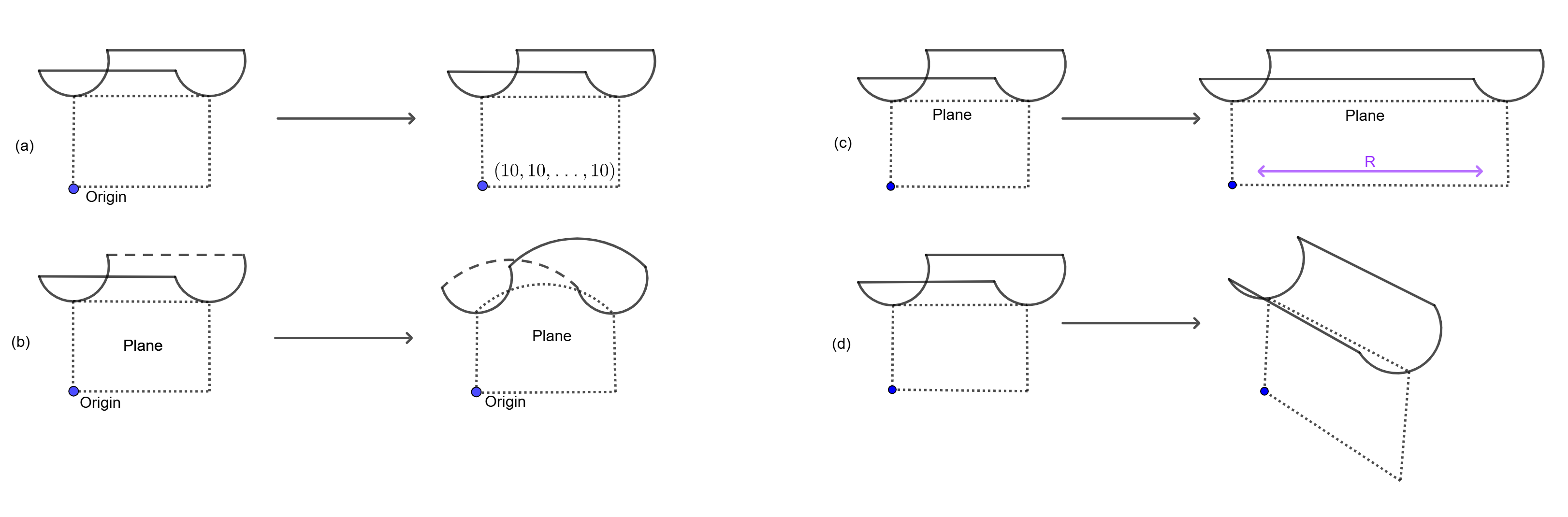}
    \caption{Admissible actions when $n=3,k=1$: (a) Translating the plane and a trough; for example, the origin is moved to $(10,10,\dots,10)$; (b) Bending the trough along the plane; (c) Stretching the trough along the plane by $R$ units; (d) Rotating the trough along with the plane.}
    \label{Actions}
\end{figure}

Using compositions of the admissible actions defined above,
one can construct a map $g_{j,k} : \re^n \mapsto \re^n $ such that the following holds true:
\begin{itemize}
    \item $g_{j,k}(D^k \times (\underbrace{0,0,\dots,0}_{n-k-1\ \text{entries}}))
= h_{j,k}(D^k)$ and
\item the plane $g_{j,k}(\{x \in \re^n \; ; \;  x_{n-1} = 0\})$ contains the plane $P.$
\end{itemize}

We will call the geometric object
$g_{j,k}(D^k \times \underline{S}^{\,n-k-1})$ a \emph{dent}. Now, we are going to use these dents to build $\pa \Omega$. Further, for any $1 \leq k \leq n-2$, one can take $\eta>0$ small so that after considering appropriate values of $R_1, R_2, R_3$ in the admissible actions, one can put a piece of the dent on $\pa B(0,1)$ so that the following holds true:
\begin{itemize}
\item For all $0 \leq k \leq n-2$ and every $1 \leq j \leq l(k)$, one has
\[
 \{ x \in \pa B(0,1) \; ; \; {\rm dist}(x, \tilde{\cone}\intersect \pa B(0,1))
\leq \eta\} \subset g_{j,k}\bigl(D^k \times \underline{S}^{\,n-k-1}\bigr),
\]
and
\[
g_{j,k}\bigl(D^k \times \underline{S}^{\,n-k-1}\bigr) \subset \{ x \in \pa B(0,1) \; ; \; {\rm dist}(x, \tilde{\cone}\intersect \pa B(0,1))
\leq 2\eta\} .
\]

\item For any $0 \leq k \leq n-2$ and any $1 \leq j \leq l(k)$, one has
\begin{equation}\label{gap}
    \begin{aligned}
        g_{j,k}(D^k \times \underline{S}^{\,n-k-1}) \cap \partial B(0,1)\subset
\left\{
y \in \pa B(0,1)
\; \middle| \;
{\rm dist}\!\left(y, h_{m,p}(D^p)\right) \geq 3 \eta, \forall m \neq j, p \neq k.
\right\},
    \end{aligned}
\end{equation}
\item The elements of $\left\{  g_{j,k}(D^k \times \underline{S}^{\,n-k-1}) \; ;\;  j \in \{ 1,2,\dots, l(k)\} \right\}$ are non-empty, for all $k$. 
\item For each $k$, the elements in $\left\{  g_{j,k}(D^k \times \underline{S}^{\,n-k-1}) \; ;\;  j \in \{ 1,2,\dots, l(k)\} \right\}$ are all at a Euclidean distance of   at least $5\eta$ from each other.
\end{itemize}
As a consequence of the admissible actions used in the construction,
one can still assume that the set $g_{j,k}(D^k \times \underline{S}^{\,n-k-1})$ meets the plane $P$
orthogonally along $h_{j,k}(D^k)$. We now provide an algorithm to connect these dents in the small remaining gaps due to \eqref{gap}, in a Lipschitz manner.
Fix an integer $k$ so that $k$ and $k-1$ are in the set $\{1,2,\dots,n-1\}$.
Suppose that on $\tilde{\cone} \cap \partial B(0,1)$ a diffeomorphic image
of a $k$-disk, say $h_1(D^{k})$ meets a diffeomorphic image of a
$(k-1)$-disk, say $h_2(D^{k-1})$.
Correspondingly, by the construction in the previous paragraph, there
exist dents
\[
g_1\bigl(D^{k} \times \underline{S}^{\,n-k-1}\bigr)
\quad\text{and}\quad
g_2\bigl(D^{k-1} \times \underline{S}^{\,n-k}\bigr),
\]
for some diffeomorphisms $g_1$ and $g_2$ such that
$g_1(D^k \times (0,\dots,0)) = h_1(D^{k})$ and $g_2(D^{k-1} \times (0,\dots,0))=h_2(D^{k-1})$. As before, assume that $P$ is one of the $\binom{n+1}{2}$ many $(n-1)$-dimensional hyperplanes
appearing in $\tilde{\cone}$ and that the hyperplane $P$ touches
$g_1\bigl(D^k \times \underline{S}^{\,n-k-1}\bigr)$ and
$g_2\bigl(D^{k-1} \times \underline{S}^{\,n-k}\bigr)$ orthogonally along $h_1(D^{k})$ and $h_2(D^{k-1})$, respectively.
Now, the crucial observation is that the manifold boundary of
$g_1\bigl(D^k \times \underline{S}^{\,n-k-1}\bigr)$
can be visualized as part of the manifold boundary of
$g_2\bigl(D^{k-1} \times \underline{S}^{\,n-k}\bigr)$.

In the gap created by \eqref{gap}, after a possible further translation of the dent
$g_1\bigl(D^k \times \underline{S}^{\,n-k-1}\bigr)$ away from the origin
along the plane $P$ followed by a suitable admissible bending, stretching and shrinking of
this dent along $P$ [all of these actions are admissible], one can attach
$g_1\bigl(D^k \times \underline{S}^{\,n-k-1}\bigr)$
to the manifold boundary of
$g_2\bigl(D^{k-1} \times \underline{S}^{\,n-k}\bigr)$ continuously, or more generally in a Lipschitz manner. Again, since the admissible actions preserve \eqref{goodnormals}, the geometric object one gets after the joining of the dents still satisfies \eqref{goodnormals}.

Additionally, suppose that there is another dent
$g_3\bigl(D^{k} \times \underline{S}^{\,n-k-1}\bigr)$
such that the corresponding $h_3(D^{k})$ meets $h_2(D^{k-1})$ on $\pa B(0,1)$ - note that $h_3(D^k)$ would not necessarily lie on the same $P \intersect \pa B(0,1)$. Then, further performing admissible actions along the hyperplane $P$, one can attach  $g_3\bigl(D^{k} \times \underline{S}^{\,n-k-1}\bigr)$ to $g_2\bigl(D^{k-1} \times \underline{S}^{\,n-k}\bigr)$ in a similar manner to the way in which $g_1\bigl(D^{k} \times \underline{S}^{\,n-k-1}\bigr)$ was joined to $g_2\bigl(D^{k-1} \times \underline{S}^{\,n-k}\bigr)$. By choosing $\eta>0$ to be even smaller, one can ensure that $g_1\bigl(D^{k} \times \underline{S}^{\,n-k-1}\bigr)$ is a positive distance away from $g_3\bigl(D^{k} \times \underline{S}^{\,n-k-1}\bigr)$. One can use a similar idea to join the dents of the form $g_4\bigl(D^{k-1} \times \underline{S}^{\,n-k}\bigr)$ to $g_5\bigl(D^{k-2} \times \underline{S}^{\,n-k+1}\bigr)$ and so on.

We remark here again that, since the condition \eqref{goodnormals} is required to hold $\mathcal{H}^{n-1}$--a.e.\ on the boundary, it is sufficient to define the domain
in a Lipschitz manner. We now have a surface that consists of dents which are joined in a Lipschitz manner, but still has some gaps outside these dents: this makes the constructed surface not a closed surface yet.
So, outside the collection of these dents, one can connect the already connected
dents in an arbitrarily smooth way, ensuring that the resulting surface remains at
a distance of at least $1/2$ from the origin.
This resulting topologically closed $(n-1)$-dimensional surface is an explicit example of the boundary of an admissible domain $\ \Omega$. The region  inside this constructed surface that contains the origin is an example of an admissible domain $\Omega$. The cone $\tilde{\cone}$ partitions the domain $\Omega$ in a natural way and henceforth, we define
\[
\cone := \tilde{\cone} \cap \Omega,
\]
and the regions of $\Omega$ partitioned by this $\cone$ are denoted by
$S_i$. We will denote the collection of these regions by
\[
\ms := (S_1, S_2, \dots,S_n,S_{n+1}).
\]
The constructed pair $(\Omega$, $\ms)$ above is therefore an explicit example that is admissible for \cref{isolated_local_minimizer}. It remains now to verify that Property \ref{property} is satisfied by the model geometries $D^k \times \underline{S}^{\,n-k-1}$ in the case when $P = \{ x_{n-1}=0\}$. This verification is carried out in the following lemma.

\begin{Lemma}\label{verification}
    Let $k \in \{0,1,\dots, n-2\}$.
In $\R^n$, the geometric object
$D^k \times \underline{S}^{\,n-k-1}$
satisfies the condition \eqref{goodnormals}
with respect to the hyperplane $\{x_{n-1}=0\}$.
That is, let $\vec{n}$ denote the unit normal to the hyperplane $\{x_{n-1}=0\}$
pointing from the half-space $\{x_{n-1}<0\}$ towards $\{x_{n-1}>0\}$.
Then, for every $ x \in D^k \times \underline{S}^{\,n-k-1}$, the following holds:
    \begin{equation}\label{verify_inequality}
    \left\{\begin{aligned}
        &\vec{n}\cdot\vec{\nu}_{D^k \times \underline{S}^{n-k-1}}(x) >0\quad\text{if}\; x_{n-1} <0,\\
        &\vec{n}\cdot \vec{\nu}_{D^k \times \underline{S}^{n-k-1}}(x) = 0 \quad\text{if}\; x_{n-1} = 0,\\
        &\vec{n}\cdot\vec{\nu}_{D^k \times \underline{S}^{n-k-1}}(x) <0\quad\text{if}\; x_{n-1} >0,
    \end{aligned}\right.
    \end{equation}
    where $x_i$ is the $i^{th}$ element of the tuple $x = (x_1,x_2,\dots, x_n)$ and $\vec{\nu}_{D^k \times \underline{S}^{n-k-1}}(x)$ denotes the unit normal to $D^k \times \underline{S}^{n-k-1}$ at $x$ in the direction towards the positive $x_n$ axis.
\end{Lemma}

\begin{proof}[Proof of Lemma \ref{verification}]
  Fix $k\geq 1$ and set the unit normal to $\{x_{n-1}=0\}$ pointing from $\{x_{n-1} <0\}$ towards $\{ x_{n-1} > 0\}$ to be
   \[
    \vec{n} = (\underbrace{0,0,\dots, 0}_{n-2 \text{ many}}, 1,0).
   \]
For a surface parametrized by the function $ \sum_{j=k+1}^n x_j^2 -1 = 0$, the normal to this surface pointing towards the positive $x_n$ axis is given by \[
   (\underbrace{0,0,\dots, 0}_{k \text{ many}}, -2x_{k+1}, \dots,-2x_{n-1}, -2x_n).
   \]This shows that $\vec{n}\cdot\vec{\nu}_{D^k \times \underline{S}^{n-k-1}}(x)= -x_{n-1}$. When $k =0$, ${D^k \times \underline{S}^{n-k-1}}$ is just the $(n-1)$-dimensional hemisphere and the corresponding normal would be $-x$. Hence, here $\vec{n}\cdot\vec{\nu}_{D^k \times \underline{S}^{n-k-1}}(x) = -x_{n-1}$. The proof of lemma is complete.
   
\end{proof}

\section{Requirements and a Proposition on the assumptions.}\label{require}
For background on sets of finite perimeter, we refer the reader to e.g., \cite{EvansGareipy,G,Maggi}. In the analysis to follow, we require the following two results.
\begin{Lemma}[cf. \cite{Zeimer}]\label{hyperplane-theorem}
   Let $A \subset \mathbb{R}^n$, $n \ge 2$, be a set of finite perimeter. Suppose there exists a point $x_0 \in \partial^* A$ with the property that, for some cube $Q_0$ centered at $x_0$, the measure-theoretic outward normal to $A$, denoted by $\nu_A(x)$, is constant: $\nu_A(x) = \nu_0$ for all $x \in Q_0 \cap \partial^* A$. Then,
\[
Q_0 \cap \partial^* A = Q_0 \cap \Pi,
\]
where $\Pi$ is the unique hyperplane in $\mathbb{R}^n$ containing $x_0$ with normal vector $\nu_0$.

\end{Lemma}
\begin{Proposition}[Relative isoperimetric inequality, cf. \cite{Abhi+Peter}]\label{isoperimetric_theorem}
Let $\Omega \subset \mathbb{R}^n$, $n \ge 2$, be a bounded and connected domain with $C^1$ boundary.  
Then, there exist constants $r_0 = r_0(\Omega) > 0$ and $C_0 = C_0(n, \Omega) > 0$ such that, for any $x_0 \in \partial \Omega$ and any set $A \subset \Omega \cap B(x_0, r_0)$ having finite perimeter, one has
\begin{equation*}
\h^{n-1}( \Omega \intersect \pa^* A) \geq C_0 \Big[ \mathcal{L}^n(A) \Big]^{\frac{n-1}{n}}.
\end{equation*}
\end{Proposition} 

Now, we show that the claim on the assumptions we made earlier is true in the following proposition.
\begin{Proposition}\label{propn_asu}
    In the case when $n \geq 3$, Assumption \ref{closeness_of_cij} implies Assumption \ref{assumption_c_ij_positive}.
\end{Proposition}
\begin{proof}[Proof of Proposition \ref{propn_asu}]\footnote[2]{Portions of this proof were suggested by ChatGPT.}
    Assume, for contradiction, that Assumption \ref{closeness_of_cij} holds and Assumption \ref{assumption_c_ij_positive} is false. Then, there is a vector $x:= (x_2,\dots, x_{n+1}) \in \re^n$ so that $F(x) \leq 0$. 
    Now, expanding $F$ in \eqref{quadratic_form} in terms of $c_i$, one gets
    \begin{equation}\label{Fpostive}
    0 \geq F(x_2,\dots, x_{n+1}) = \sum_{i,j=2}^{n+1}[c_1^2 + c_1(c_i + c_j) - c_ic_j] \;x_ix_j=: x^T(D\tilde{H}D)\;x ,
    \end{equation}
    where $A^T$ denotes the transpose of $A$ and the matrix $D := {\rm diagonal}(c_2, c_3,\dots, c_{n+1})$ is an $n\times n$ matrix whose $(i,i)^{th}$ entry is $c_{i+1}$ and $0$ elsewhere. Here, the $(i,j)^{th}$ entry of the $n\times n$ matrix $\tilde{H}$ is given by
    \[
    \left(\frac{c_1+c_{i+1}}{c_{i+1}}\right)\left(\frac{c_1 + c_{j+1}}{c_{j+1}}\right) -2 (1 - \delta_{ij}),
    \]

where $\delta_{ij}$ is the Kronecker delta function. So, the condition \eqref{Fpostive} is equivalent to saying $y^T\tilde{H}y\leq 0$ for $y := Dx$. Letting $\omega_i := \frac{c_1 + c_i}{c_i}$, the matrix $\tilde{H}$ turns out to be 
\[
\tilde{H} = H_1 - 2\; \mathbbm{1}\mathbbm{1}^T, \; \text{ where } \; H_1:= 2I + \omega^T \omega,
\]
where $\omega := (\omega_2,\dots, \omega_{n+1})\in \re^n$ and $\mathbbm{1} := (1,1,\dots,1) \in \re^n$.  So, $y^T\tilde{H}y \leq 0$ is equivalent to saying
\begin{equation}\label{cstoapply}
    \frac{1}{2}y^T H_1 y \leq ( \mathbbm{1} \cdot y )^2.
\end{equation}

By the definition of $H_1$, one sees that $H_1$ is positive definite and symmetric. So, $H_1^{\pm\frac{ 1}{2}}$ exist, and these matrices are also symmetric. Indeed, by the spectral theorem, there exist an orthogonal matrix $Q$ and a set of positive eigenvalues $\{\lambda_1,\lambda_2,\dots,\lambda_n\}$ so that 
\[
H_1 = Q^T \cdot \text{\rm diag}(\lambda_1,\lambda_2,\dots,\lambda_n)\cdot Q.
\]
Thus, explicitly, one has
\[
H_1^{\pm \frac{1}{2}} = Q^T \cdot \text{\rm diag}(\lambda_1^{\pm \frac{1}{2}},\lambda_2^{\pm \frac{1}{2}},\dots,\lambda_n^{\pm \frac{1}{2}})\cdot Q.
\]
Moreover, one has 
\[
H_1^{-\frac{1}{2}}H_1^{+\frac{1}{2}}= {\rm Id}_n, 
\]
where Id$_n$ is the  $n\times n$ identity matrix. Applying the Cauchy-Schwarz inequality to \eqref{cstoapply}, one obtains
\begin{equation}\label{csapplied}
    \begin{aligned}
        \frac{1}{2}y^T H_1 y\leq (\mathbbm{1} \cdot y )^2 
        = (\mathbbm{1}^Ty)^2 
        &= \left[ \mathbbm{1}^T H_1^{-\frac{1}{2}}H_1^{+\frac{1}{2}} y\right]^2\\
         &= \left[ \left(H_1^{-\frac{1}{2}}\mathbbm{1}\right)^T H_1^{+\frac{1}{2}} y\right]^2\\
        &= \left[H_1^{-\frac{1}{2}} \mathbbm{1} \cdot H_1^{+\frac{1}{2}} y\right]^2 
        \\
        &\leq \left| H_1^{-\frac{1}{2}} \mathbbm{1}\right|^2 \left| H_1^{+\frac{1}{2}} y\right|^2\\
        &= \left[ \left( H_1^{-\frac{1}{2}} \mathbbm{1}\right)^T\left( H_1^{-\frac{1}{2}} \mathbbm{1}\right)\right]\left[\left( H_1^{+\frac{1}{2}} y\right)^T\left( H_1^{+\frac{1}{2}} y\right)\right]
        \\ &= [ \mathbbm{1}^T H_1^{-1} \mathbbm{1}]\;[ y^T H_1 y].
    \end{aligned}
\end{equation}
Noting that $H_1$ is positive definite, one gets
\begin{equation}\label{before_SM}
    \frac{1}{2} \leq\mathbbm{1}^T H_{1}^{-1} \mathbbm{1}.
\end{equation}
The Sherman-Morrison formula (cf. \cite[Theorem 2.3.10]{matrixbook}) gives 
\[
H_{1}^{-1} = \frac{1}{2}\; {\rm Id}_n - \frac{\omega\omega^T}{2(2+|\omega|^2)}.
\]
And so, one has
\[
\mathbbm{1}^T H_1^{-1}\mathbbm{1} = \frac{n}{2} - \frac{\left(\sum_{i=2}^{n+1} \omega_i\right)^2}{2\left(2+\left(\sum_{i=2}^{n+1} \omega_i^2\right)\right)}.
\]
Using this estimate in \eqref{before_SM}, one gets that
\begin{equation}\label{star}
\left(\sum_{i=2}^{n+1} \omega_i\right)^2 - (n-1) \left(\sum_{i=2}^{n+1} \omega_i^2\right)\leq 2(n-1).
\end{equation}
Now, Assumption \ref{closeness_of_cij} implies that for every $i = 2,3,\dots, n+1$, one has
\[
\frac{n-2}{n} < \frac{c_1}{c_i} < \frac{n}{n-2}.
\]
Or equivalently, using  $\omega_i = \frac{c_1 + c_i}{c_i}$, one has
\[
\frac{2(n-1)}{n}= 1 + \frac{n-2}{n}< \omega_i < 1 + \frac{n}{n-2} = \frac{2(n-1)}{n-2}.
\]
Based on \eqref{star}, for $\tau := (\tau_2,\dots, \tau_{n+1})\in \re^n$, we define a new function to analyze:
\[
\Phi(\tau) := \left( \sum_{i=2}^{n+1} \tau_i\right)^2 - (n-1) \sum_{i=2}^{n+1} \tau_i^2.
\]
We now claim that for all $\tau$ in the $n$-box $\left[ \frac{2(n-1)}{n}, \frac{2(n-1)}{n-2}\right]^n$, one has
\begin{equation}\label{claim}
\Phi(\tau) >2(n-1),
\end{equation}
which clearly contradicts \eqref{star}. We now complete the proof of the Proposition \eqref{propn_asu} by proving the claim leading to \eqref{claim}. Indeed, one sees that $\Phi$ is concave in each variable $\tau_i$ separately,  with all the other variables fixed, since
\[
\left( \frac{\pa}{\pa \tau_i}\right)^2 \Phi(\tau) = 4-2n<0, \; \forall\; i = 2,3,\dots,n+1.
\]
Hence, the minimum of $\Phi$ must occur at a vertex of the $n$-box $\left[ \frac{2(n-1)}{n}, \frac{2(n-1)}{n-2}\right]^n$. So, in this $n$-box, assume that $\Phi$ attains its minimum in the $n$-box when $\tilde{l}$ many $\tau_i$ are $\frac{2(n-1)}{n}$ and that $n-\tilde{l}$ many $\tau_i$ are $\frac{2(n-1)}{n-2}$. Then, the minimum  of $\Phi$ in the $n$-box,x denoted by $\Phi_{min}$ is
\[
\Phi_{min} = \left[ \tilde{l} \left(\frac{2(n-1)}{n}\right) + (n-\tilde{l})\left(\frac{2(n-1)}{n-2}\right) \right]^2- (n-1) \left[ \tilde{l} \left(\frac{2(n-1)}{n}\right)^2 + (n-\tilde{l} )\left(\frac{2(n-1)}{n-2} \right)^2\right]
\]
Simplifying the above expression and noting that for any $\tilde{l} \in\{1,2,\dots,n\}$, the quantity $(n-\tilde{l} )(n-\tilde{l} -1) \geq0$, one obtains:
\begin{equation}\label{propn_concluding}
\begin{aligned}
\Phi_{min} 
&= \frac{4(n-1)^2}{n^2(n-2)^2} \left[ n (n-2)^2+4(n-\tilde{l} )(n-\tilde{l} -1)\right]\\
&\geq \frac{4(n-1)^2}{n^2(n-2)^2} \left[n (n-2)^2\right] = \frac{4(n-1)^2}{n} > 2(n-1). 
\end{aligned}
\end{equation}
The proof of the proposition is now complete.
\end{proof}
\begin{Remark}
   The converse of Proposition \ref{propn_asu} is false: Assumption \ref{assumption_c_ij_positive} does not always imply Assumption \ref{closeness_of_cij}. Indeed, one can consider the specific set of values for the constants $c_i$ as follows: $c_2=1, c_3 = c_4 = \dots = c_{n+1}=0$ and $c_1$ is such that $c_1^2 + 2c_1 - 1 >0$. Then, in this case, the quadratic form turns out to be
    \[
    F(x_2,\dots,x_{n+1}) = [ c_1^2 + 2c_1 - 1] x_2^2 \geq 0 \; , \; \text{ for any }(x_2,\dots,x_{n+1}) \in \re^n.
    \]
  So, Assumption \ref{assumption_c_ij_positive} holds true. However, one sees that
    \[
    \frac{c_k}{c_1}= 0 < \frac{n-2}{n}\; , \; \text{ for any }k \in \{2,3,\dots, n+1\}.
    \]
    which shows that Assumption \ref{closeness_of_cij} is not true in this case. 
\end{Remark}

\section{Proof of the main theorem \ref{isolated_local_minimizer}.}
Fix a domain $\Omega \subset \R^n$ and the partition of it given by $\ms = (S_1,S_2, \dots, S_{n+1})$ such that they satisfy Property \ref{property}. For example, one can consider $\Omega$ and $\ms$ as defined in Section \ref{Construct} above. The goal now is to show that Theorem \ref{isolated_local_minimizer} holds true for this domain $\Omega$ and the partition $\ms$ of it. Proceeding further, for this configuration of $\Omega$ and $\ms$, for any $\delta>0$, we look at the following constrained partition problem:
\begin{equation}\label{conprob}
    \inf \left\{E_0(\mt,\Omega)\; | \;\mt\in\mathcal{A}^\delta\right\}, 
\end{equation}
where the admissible class $\mathcal{A}^{\delta}$ is given by \begin{equation}\label{admissible_class}
    \mathcal{A}^\delta := \{ \mt:=( T_1,T_2,\dots,T_{n+1}) \text{ is a partition of }\Omega \text{ such that } \lt \left( \mt \Delta \ms\right) \leq \delta\}.
    \end{equation}
As a consequence of the lower semicontinuity property of perimeter under $L^1$-convergence, there exists a minimizer in $\mathcal{A}^{\delta}$ to \eqref{conprob} which we will call
\[\mt^{\delta}= (T_1^\delta,T_2^\delta,\dots,T_{n+1}^\delta).
\]

 At this point, we highlight that the proof of Theorem \ref{lemma1} closely parallels that of \cite[Theorem 3.2]{Abhi+Peter}. In the proof of part \ref{3.2.1} of Theorem \ref{lemma1}, the role of four partitioned regions in \cite[Theorem 3.2 (I)]{Abhi+Peter} is played here by $n+1$ regions and the rest of the proof is completed by induction. For the sake of completeness and to give intuition for part \ref{3.2.2} of Theorem \ref{lemma1}, we briefly mention the proof of part \ref{3.2.1}. Regarding the next part, we recall that the proof of \cite[Theorem 3.2 (III)]{Abhi+Peter} shows that it is not efficient in terms of $E_0$ to have $T_3$ or $T_4$ around $\pa \Omega$ where $\pa S_1$ meets $\pa S_2$. The same idea is used below in an inductive way to prove part \ref{3.2.2} of Theorem \ref{lemma1}. 

\begin{Theorem}[see also \cite{Abhi+Peter}, \cite{Maggi}, \cite{Leonardi}]\label{lemma1}
 Let $\eta = \eta(\Omega)$ be the number defined in Property \ref{property}. Then, there exists a number $\delta = \delta(\cone,\Omega)>0$ such that the following conditions hold for the partition $\mt^\delta$:
\begin{enumerate}[(I)]
    \item\label{3.2.1} \label{thm4.1}For each $i \in \{ 1,2,\dots, n+1\}$, one has:
    \[
    \h^{n-1}\bigg(\Big\{ x \in \big( \bigcup_{\substack{j\in\{1,2,\dots, n+1\},\\j\not=i}} T_j^\delta\big) \intersect \pa \Omega \intersect \pa S_i \; ; \; {\rm dist}(x,\cone) \geq \eta \bigg\} \bigg) = 0.
\]
    \item \label{3.2.2}
     For any $k \in \{2,3,\dots,n\}$, the $\h^{n-1}$ measure of the set
    \begin{equation*}
   \left\{  
    \begin{aligned}
        & x \in  \intersect_{i_j \in I_k} \pa S_{i_j} \intersect \pa \Omega \; \text{\rm such that } \\
        &{\rm dist}(x,\mathcal{C}) < \eta\; \text{\rm and }  {\rm dist}(x, \cup_{i_j \notin I_k} \pa S_{i_j}) > \frac{\eta}{2}
    \end{aligned} \right\} \bigcap \left(\cup_{i_j \notin I_k}  T_{i_j}^{\delta}\right)
    \end{equation*}
    is zero, where $I_k := \{i_1, i_2, \dots, i_{k-1}, i_k\}$ with distinct $i_j \in \{ 1,2,\dots,n,  n+1\}$.
\end{enumerate}

\end{Theorem}

\begin{proof}[Proof of Theorem \ref{lemma1}]
Throughout this proof we will suppress the dependence of $\mtd$ on $\delta$ and simply write it as $\mt=(T_1,T_2,\dots, T_n,T_{n+1})$. Further, we will adopt the formulation \eqref{formulation_1} of $E_0(.,\Omega)$ for this proof. Also, at this point we introduce simplifying notation as follows:
 \begin{equation*}
\begin{aligned}
    t_{i,j}^s(x_0)&:=\h^{n-1}\big(\pa^*T_i\cap\pa^*T_j\cap B(x_0,s)\big) \text{ for }x_0 \in \pa \Omega,\\
    c_{max} &:= \max_{1 \leq j \leq n+1} c_j \text{ and }c_{min} := \min_{1 \leq j \leq n+1} c_j.
    \end{aligned}
\end{equation*}
\begin{description}
    \item[Proof of \ref{3.2.1}]
For $r_0$ from Proposition \ref{isoperimetric_theorem}, choose any $r_1 \in(0,r_0)$ such that for any $i \in \{1,2,\dots,n+1\}$ and for any $x \in \pa S_i \intersect \pa \Omega \text{ with } {\rm dist}(x,\cone)\geq \eta$, one has
\begin{equation*}
B(x,r_1)\intersect \Omega \ssubset S_i,
\end{equation*}
 For ease of computation, we will consider the case $i=1$; the other cases follow similarly. Fix an $x_0 \in \pa S_1 \intersect\pa \Omega$ such that ${\rm dist}(x_0,\cone) \geq \eta$. Now, depending on densities, we have two cases. 
\begin{description}
    \item[Case 1] Suppose first that for some $0 <r\leq r_1$ one has the condition
    \begin{equation}\label{case1hype}
    \lt \left(\left[\cup_{j=2}^{n+1}T_j\right] \intersect B(x_0,r)\right) \leq \varepsilon_0 r^n, \text{ where } \varepsilon_0 = \left(\frac{\Lambda}{2n}\right)^n,
    \end{equation}
     where the constant $\Lambda$ is any fixed positive number less than $\left(\frac{C_0 \; c_{min}}{c_{max} + c_{min}}\right)$ with $C_0$ defined in Proposition \ref{isoperimetric_theorem}.
     Then we will show that 
    \begin{equation}\label{case1conclu}
    \lt\left(\left[\cup_{j=2}^{n+1}T_j\right]\intersect B\left(x_0,\frac{r}{2}\right)\right) = 0.
    \end{equation}
To begin, suppressing dependence on $\delta$, we let
\begin{equation*}
m(s) := \lt \big( \left[\cup_{j=2}^{n+1}T_j\right] \intersect B(x_0,s)\big).
\end{equation*}

For each $0 < s < r_1$, we create a new partition by giving $\left[\cup_{j=2}^{n+1}T_j\right] \intersect B(x_0,s)$ to $T_1$. More precisely, we define a new partition $\mf^s := (F_1^s, F_2^s, \dots, F_{n+1}^s)$ by
\begin{equation*}
    \begin{aligned}
        F_1^s &:= T_1 \union \left[B(x_0,s) \intersect \left(\cup_{j=2}^{n+1}T_j\right)\right] , \\
        F_2^s &:= T_2 \setminus B(x_0,s), \\
       \vdots& \\
        F_{n+1}^s &:= T_{n+1} \setminus B(x_0,s),\\
    \end{aligned}
\end{equation*}
where we again have suppressed the dependence of $\mf^s$ on $\delta$: see \cref{unsergerised} and \cref{surgersised}. At this point, we note that the newly created partition $\mf^s$ is an element of $\mathcal{A}^\delta$ and as $\mt = (T_1,T_2,\dots, T_{n+1})$ is assumed to be a minimizer of \eqref{conprob}, we have the minimality condition:
\begin{equation}\label{minimality_1}
E_0(\mt,\Omega)\leq E_0(\mf^s,\Omega)\quad\text{for all}\;s\in (0,r_1).
\
\end{equation}

\begin{figure}
    \centering
    \includegraphics[width=0.5\linewidth]{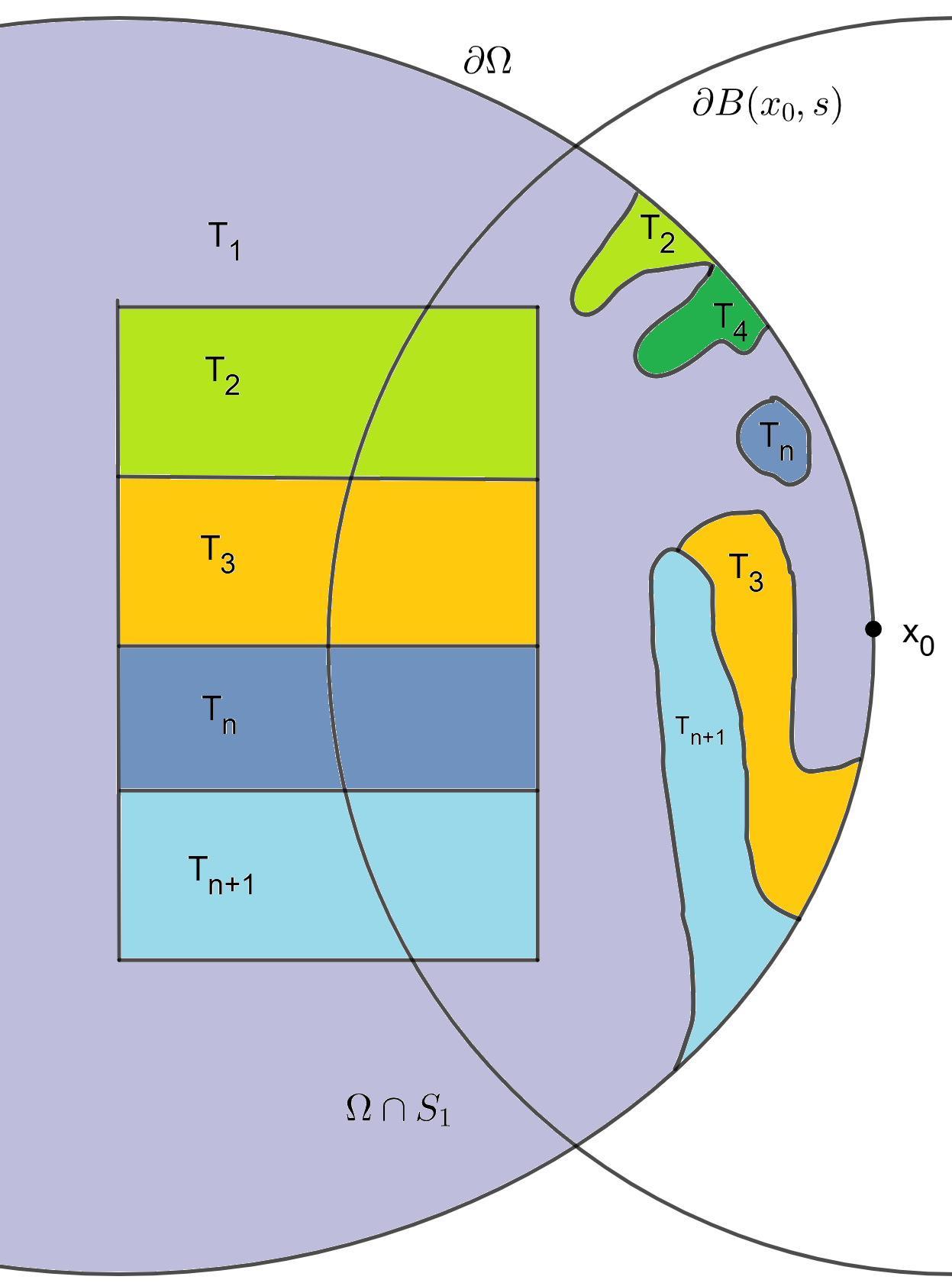}
    \caption{An example of the partition $\mt = (T_1,T_2,\dots  ,T_{n+1})$ in $S_1 \intersect B(x_0,s)$.}
    \label{unsergerised}
\end{figure}

\begin{figure}
    \centering
    \includegraphics[width=0.5\linewidth]{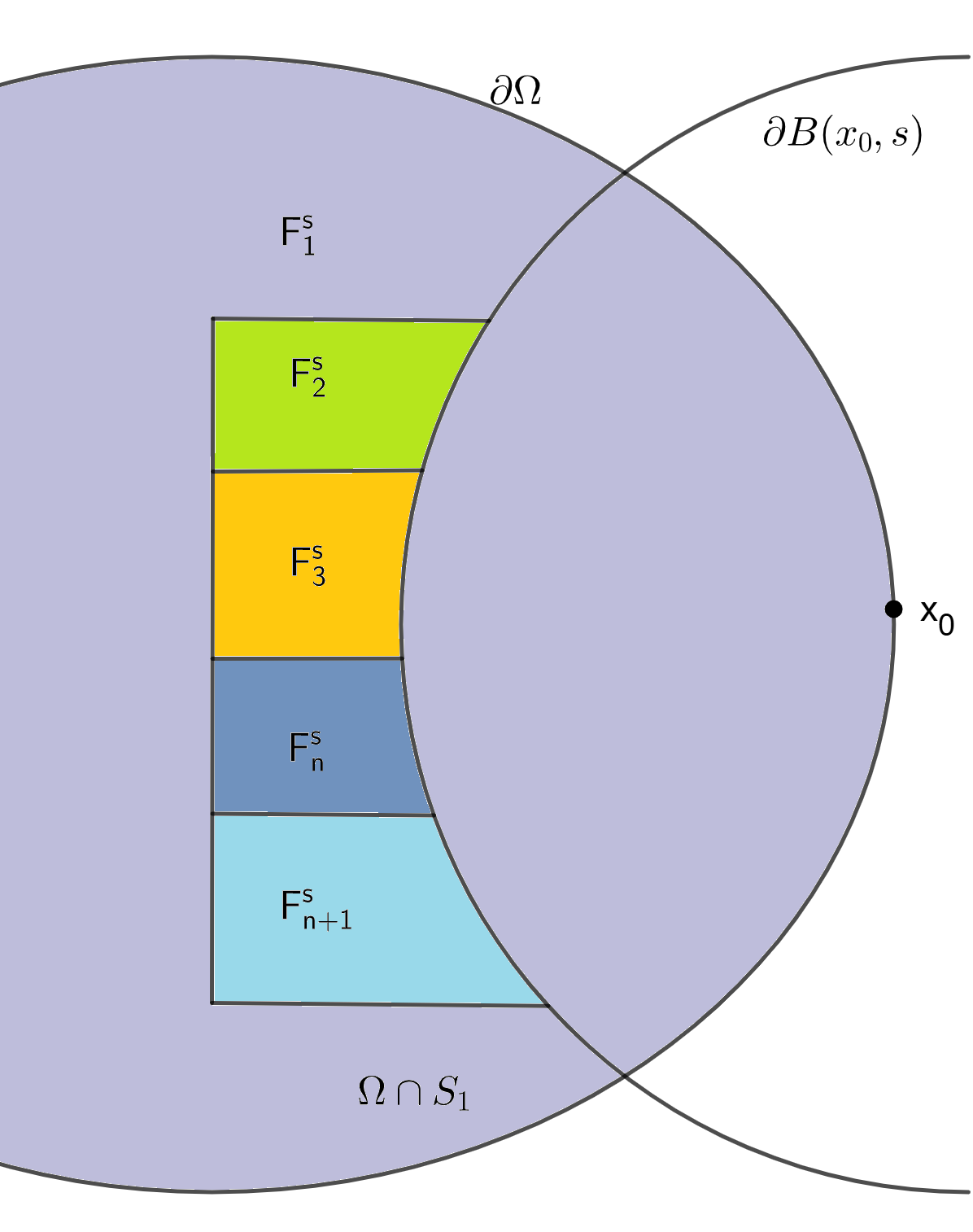}
    \caption{The corresponding modification of the partition from \cref{unsergerised}, given by $\mf^s = (F_1^s,F_2^s,\dots , F_{n+1}^s)$ in $S_1 \intersect B(x_0,s)$.}
    \label{surgersised}
\end{figure}

Following the proof of \cite[Theorem 3.2(I)]{Abhi+Peter}, we compute both sides of \eqref{minimality_1} and use the isoperimetric inequality (Proposition \ref{isoperimetric_theorem}) to obtain
\begin{equation}\label{lambda_1}
    \Lambda m(s)^{\frac{n-1}{n}}\leq m'(s),
\end{equation}
where $\Lambda$ is given in \eqref{case1hype}. One can now integrate \eqref{lambda_1} on the support of $m$ and as in the proof of \cite[Theorem 3.2(I)]{Abhi+Peter}, one can conclude the identity \eqref{case1conclu}.

\item[Case 2] Suppose now that for any $x_0 \in \pa S_1 \intersect \pa \Omega$ such that ${\rm dist}(x_0,\cone) \geq \eta$, the assumption \eqref{case1hype} does not hold. That is, suppose that
    \begin{equation}\label{counterlem1}
      \forall r \leq r_1, \quad  \lt(\left[\union_{j=2}^{n+1}T_j\right]\intersect B(x_0,r)) > \varepsilon_0 r^n.
    \end{equation}
 Let $\delta$ satisfy:
 \begin{equation}\label{Karthik}
 0<\delta < \min\left\{\left(\frac{r_1}{2}\right)^{n+1}, \frac{\varepsilon_0^{n+1}}{2}\right\}.
 \end{equation}
 For this $\delta$, plug in $r = \delta^{\frac{1}{n+1}}$ in \eqref{counterlem1} to get
 \begin{equation}\label{AAA}
      \lt(\left[\union_{j=2}^{n+1}T_j\right] \intersect B(x_0,r_1)) > \varepsilon_0 \delta^{\frac{n}{n+1}}.
 \end{equation}
 From the definition of the admissible class \eqref{admissible_class}, one has $\lt(\mt \Delta \ms) \leq \delta$ and so, we have
\begin{equation}\label{BBB}
    \delta \geq \lt(\left[\union_{j=2}^{n+1}T_j\right] \intersect B(x_0,r_1)).
\end{equation}
Thus, combining \eqref{AAA}, \eqref{BBB} and using \eqref{Karthik}, one sees that \eqref{counterlem1} cannot hold true.
\end{description}
The proof of \ref{3.2.1} is now complete.

\item [Proof of \ref{3.2.2}]
Fix $k \in \{2,3,\dots, n\}$ and a collection $I_k = \{i_1, i_2,\dots ,i_k\}$ as mentioned in the hypothesis. Further, let $x_0$ be in a connected component of
\begin{equation}\label{components}
\left\{  
    \begin{aligned}
        & x \in  \intersect_{i_j \in I_k} \pa S_{i_j} \intersect \pa \Omega \; \text{ such that } {\rm dist}(x,\mathcal{C}) < \eta \text{ and }  {\rm dist}(x,\cup_{i_j \notin I_k}  \pa S_{i_j}) > \frac{\eta}{2}.
    \end{aligned} \right\} .
    \end{equation}
For $r_0$ from Proposition \ref{isoperimetric_theorem}, let $r_2 = r_2(I_k) \in (0,r_0)$ be a number that is less than half the Euclidean distance between the distinct connected components of \eqref{components}. The goal now is to show that for any $i_0 \in \{1,2,\dots, n+1\}\setminus I_k$, $T_{i_0}$ does not meet a neighborhood of $x_0$ in $\bar{\Omega}$.

In analogy with \eqref{case1hype} from the proof of \ref{3.2.1}, we first assume that for some $0<r<r_2$, one has
\begin{equation}\label{X1}
\lt \left( B(x_0,r) \intersect T_{i_0}\right) \leq \varepsilon_0 r^n, \text{ with } \varepsilon_0=\left(\frac{\Lambda}{2n}\right)^n,
\end{equation}
where $\Lambda$ is a positive number to be determined. By an application of the pigeonhole principle, we see that for every $0<s<r$, there exists $j_0 = j_0(s) \in I_k$ such that one has
\begin{equation}\label{pigeonhole}
t_{j_0 ,i_0}^s(x_0)\geq \frac{1}{n}\sum_{\substack{1\leq p_0 \leq n+1,\\ p_0 \neq i_0}}t_{p_0, i_0}^s(x_0).
\end{equation}
Now, we create a competing partition by giving $ B(x_0,s)\intersect T_{i_0}$ to $T_{j_0}$, rather than giving all of $B(x_0,s) \intersect (\cup_{\substack{1\leq p_0 \leq n+1,\\ p_0 \neq i_0}} T_{p_0})$ to $T_{j_0}$. More precisely, for all $0 < s < r$, we define a new partition $\mathcal{G}^s \in \mathcal{A}^{\delta}$ by
    \begin{equation}\label{part_G}
        \begin{aligned}
            G_{j_0}^s &:= T_{j_0}\union \Big[ B(x_0,s)\intersect T_{i_0} \Big], \\
            G_k^s &:= T_k, \text{ if } k \in\{1,2,\dots, n+1\}\text{ and } k \neq j_0, \\
            G_{i_0}^s &:= T_{i_0} \setminus  B(x_0,s), \\
            \mathcal{G}^s &:= (G_1^s,G_2^s,\dots,G_{n+1}^s),
        \end{aligned}
    \end{equation}
    where we again have suppressed the dependence of $\mathcal{G}^s$ on $\delta$. Here, we follow the proof of \cite[Theorem 3.2 (II),(III)]{Abhi+Peter}. For simplicity of computation, fix the value of $s$, and for this $s$, let $i_0=n+1, j_0(s) = 1$. Since $x_0$ is fixed for the current argument, we will suppress the dependence on $x_0$ and simply denote $t_{i,j}^s(x_0)$ by $t_{i,j}^s$. As in \cite{Abhi+Peter}, one can compute both sides of 
    \begin{equation}\label{minim_111}
    E_0(\mt,\Omega) \leq E_0(\mathcal{G}^s,\Omega), \quad \forall s \in ( 0, r_2),
    \end{equation}
    to obtain
    \begin{equation}\label{case2hype2}
    c_1 t_{1,n+1}^s + c_{n+1} {\rm Per}(T_{n+1} \intersect B(x_0,s); \Omega) \leq c_1 \sum_{j=2}^{n+1}t_{j,n+1}^s + [c_1 + 2c_{n+1}] p'(s),
    \end{equation}
    where the analogue of $m(s)$ here is $p(s)$ and is given by
    \[
p(s) := \mathcal{L}^n(T_{n+1} \intersect B(x_0,s)).
\]
  We recall once again here that the goal now is to show the following implication:
    \begin{equation}\label{goal}
    \lt \left(T_{n+1} \intersect B(x_0,r)\right) \leq \varepsilon_0 r^n
  \implies \lt\Big(T_{n+1}\intersect B\big(x_0,\frac{r}{2}\big)\Big) = 0,
    \end{equation}
    where $\varepsilon_0 = \left(\frac{\Lambda}{2n}\right)^n$. 
By the pigeonhole principle (cf. \eqref{pigeonhole}), one has
\[
\sum_{j=2}^n t_{j,n+1}^s - t_{1,n+1}^s \leq\frac{n-2}{n} \left(\sum_{j=2}^n t_{j,n+1}^s + t_{1,n+1}^s\right).
\]
Now, using \eqref{case2hype2}, one sees that
\begin{equation*}
\begin{aligned}
     &\quad c_{n+1}{\rm Per}(T_{n+1} \intersect B(x_0,s);\Omega)\\ &\leq c_1\left( \sum_{j=2}^n t_{j,n+1}^s - t_{1,n+1}^s\right)+ [c_1 + 2 c_{n+1}]p'(s)\\
     &\leq \frac{c_1(n-2)}{n} \left(\sum_{j=2}^n t_{j,n+1}^s + t_{1,n+1}^s+p'(s)\right)+ \left[c_1\left(1 - \frac{(n-2)}{n}\right) + 2 c_{n+1}\right]p'(s)\\
     &\leq \frac{c_1(n-2)}{n} {\rm Per}(T_{n+1}\intersect B(x_0,s);\Omega) + \left[\frac{2c_1}{n} +2 c_{n+1}\right]p'(s),
\end{aligned}
\end{equation*}
where we have used the fact that ${\rm Per}(T_{n+1} \intersect B(x_0,s); \Omega) = \sum_{j=1}^{n}t_{j,n+1}^s + p'(s)$. Grouping like terms, one has
\begin{equation}\label{appl_iso}
\left[nc_{n+1} - c_1\left( n-2\right)\right] {\rm Per}(T_{n+1} \intersect B(x_0,s); \Omega) \leq (2c_1 + 2nc_{n+1}) p'(s).
\end{equation}
 We remark here that the L.H.S. of \eqref{appl_iso} is strictly positive due to Assumption \ref{closeness_of_cij}. Now, as in the previous case, one can apply the isoperimetric inequality to \eqref{appl_iso} to get the differential inequality:
\begin{equation}\label{blah}
\begin{aligned}
    \Lambda p(s)^{\frac{n-1}{n}} \leq  p'(s),
\end{aligned}
\end{equation}
where, for $C_0$ taken from Proposition \ref{isoperimetric_theorem}, $\Lambda$ is any positive number such that
\begin{equation}\label{defn_of_lambda}
    0<\Lambda < C_0 \left[\frac{nc_{min}- (n-2)c_{max}}{(2n+2)c_{max}} \right].
\end{equation} 
 In the next step, as in the analysis following \eqref{lambda_1}, one can integrate the differential inequality \eqref{blah} on the support of $p$ to show that the implication \eqref{goal} is true in this case.

Finally, as seen earlier in the proof of case 2 of part \ref{3.2.1}, one sees that the negation of \eqref{X1} cannot be true by taking
\[
0<\delta < \min\left\{\left(\frac{r_2}{2}\right)^{n+1}, \frac{\varepsilon_0^{n+1}}{2}\right\}.
\]

We have now shown that for an $i_0 \in\{1,2,\dots,n+1\}\setminus I_k$, the set $T_{i_0}$ does not meet a neighborhood of $x_0$ in $\bar{\Omega}$. Next, choose any $ k_0 \in \{1,2,\dots,n+1\}\setminus (I_k \union \{i_0\})$ and for this $k_0$, the goal is to show that $T_{k_0}$ does not interact with a neighborhood of $x_0$ in $\bar{\Omega}$.

That is, by the pigeonhole principle, there exists $l_0 = l_0(s) \in I_k$ such that one has
\begin{equation*}
t_{l_0, k_0}^s(x_0)\geq\frac{1}{n} \sum_{\substack{1\leq p_0 \leq n+1,\\ p_0 \neq k_0}}t_{p_0,k_0}^s(x_0).
\end{equation*}
As before, for a fixed $s$, one can construct a new partition by giving  $B(x_0,s) \intersect T_{k_0}$ to $T_{l_0}$ in the ball $B(x_0,s)$. Computing both sides of the analogue of the inequality \eqref{minimality_1} here, one gets the corresponding differential inequality \eqref{blah} in this case as
    \[
     \Lambda \tilde{p}(s)^{\frac{n-1}{n}} \leq \frac{d\tilde{p}}{ds}(s),
    \]
    where $\Lambda$ is the number chosen to be as in \eqref{defn_of_lambda} and the role of $m(s)$ or $p(s)$ from the earlier steps is replaced here by $\tilde{p}(s)$ which is defined as\[
\tilde{p}(s) := \mathcal{L}^n(T_{l_0} \intersect B(x_0,s)).
\]

The rest of the proof is similar to the earlier case. Using induction, in a similar way, one can show that $T_{q_0}$ does not meet a neighborhood of $x_0$ in $\bar{\Omega}$, for any $q_0 \in \{ 1,2,\dots, n+1\} \setminus I_k$. 

\end{description}

Finally, we combine these observations to conclude Theorem \ref{lemma1} for
\begin{equation}\label{delta_small}
0<\delta < \min\left\{\left(\frac{\tilde{r}}{2}\right)^{n+1}, \frac{1}{2}\left(\frac{\Lambda}{2n}\right)^{n}\right\},
\end{equation}
where $\tilde{r}$ and $\Lambda$ are any positive fixed numbers such that
\begin{equation*}
    \begin{aligned}
        &0<\tilde{r}< \min\left\{ r_2(I_k), r_1 \; ; \; k \in\{2,3,\dots, n\}, I_k \subset\{1,2,\dots, n+1\}.
        \right\},\\
        \text{ and }&0 < \Lambda < \min\left\{ \frac{C_0c_{min}}{c_{max}+c_{min}}, C_0 \left[\frac{nc_{min}- (n-2)c_{max}}{(2n+2)c_{max}} \right]\right\}.
    \end{aligned}
\end{equation*}

The proof of Theorem \ref{lemma1} is now complete.
\end{proof}

We have now established that, for $\delta$ as in \eqref{delta_small}, if $\mathcal{T}^{\delta}$ meets $\ms$ on $\pa\Omega$, it does so only on the dented part of $\pa\Omega$, within the $\eta$-neighborhood of the simplex cone where Property \ref{property} holds true. Motivated again by the ideas in \cite[Proof of Theorem 3.1]{Abhi+Peter} (see also \cite{Zeimer,Morgan}), we will now appeal to a calibration argument to complete the proof of Theorem \ref{isolated_local_minimizer}. The main difference from \cite[Proof of Theorem 3.1]{Abhi+Peter} is that below, we apply the Gauss-Green theorem to a set of $n+1$ integrals to obtain \eqref{boundary=inside}, while in \cite{Abhi+Peter}, we apply the Gauss-Green theorem to $4$ integrals. The rest of the proof from \eqref{boundary=inside} is the same as in \cite{Abhi+Peter}, but for the sake of completeness, we provide some relevant computations from thereon.

\begin{proof}[Proof of Theorem \ref{isolated_local_minimizer}] 
We first recall that $\vec{n}_{i,j}$ denotes the unit normal to $\pa S_i \intersect \pa S_j$ pointing from $S_i$ to $S_j$, $\vec{\nu}_{i,j}$ denotes the unit normal to $\pa^* T_i \intersect \pa^* T_j$ pointing from $T_i$ to $T_j$ and $\vec{\nu}_{\pa \Omega}(x)$ denotes the unit outward normal to $\pa\Omega$ at $x \in \pa \Omega$. Here, we will adopt the formulation \eqref{formulation_2} of $E_0(.,\Omega)$ instead of \eqref{formulation_1}. The first step in this proof is to apply the Gauss-Green theorem to the three integrals mentioned below: 
\begin{equation}\label{Q23}
    \begin{aligned}
        0 &= \int_{T_1} {\rm div}(\vec{n}_{1,n+1}) \,dx = \sum_{k=2}^{n+1}\int_{\Omega\intersect\pa^*T_1 \intersect \pa^* T_k }\vec{n}_{1,n+1}\cdot \vec{\nu}_{1,k} \,d\h^{n-1}  +  \int_{\pa \Omega \intersect \pa^*T_1} \vec{n}_{1,n+1}\cdot \vec{\nu}_{\pa\Omega} \,d\h^{n-1}, 
    \end{aligned}
\end{equation}
\begin{equation}\label{Q22}
    \begin{aligned}
         0 &= \int_{T_2} {\rm div}(\vec{n}_{1,n+1}) \,dx = \sum_{\substack{k=1,\\k\neq 2}}^{n+1}\int_{\Omega\intersect\pa^*T_2 \intersect \pa^* T_k}\vec{n}_{1,n+1}\cdot \vec{\nu}_{2,k} \,d\h^{n-1}  +  \int_{\pa \Omega \intersect \pa^*T_2} \vec{n}_{1,n+1}\cdot \vec{\nu}_{\pa\Omega} \,d\h^{n-1}, 
    \end{aligned}
\end{equation}
and
\begin{equation}\label{Q21}
    \begin{aligned}
        0 &= \int_{T_2} {\rm div}(\vec{n}_{2,1}) \,dx = \sum_{\substack{k=1,\\k\neq 2}}^{n+1}\int_{\Omega\intersect \pa^*T_2 \intersect \pa^* T_k}\vec{n}_{2,1}\cdot \vec{\nu}_{2,k} \,d\h^{n-1}  +  \int_{\pa \Omega \intersect \pa^*T_1} \vec{n}_{2,1}\cdot \vec{\nu}_{\pa\Omega} \,d\h^{n-1}. 
    \end{aligned}
\end{equation}

We multiply $\vec{n}_{i,j}$ by $c_{ij}$ in equations \eqref{Q23}, \eqref{Q22} and \eqref{Q21} and adds all three equations. Using the cyclic property of the normals \eqref{gencij} and the fact that $\vec{\nu}_{i,j} = - \vec{\nu}_{j,i}$ and $\vec{n}_{i,j} = - \vec{n}_{j,i}$, one obtains
\begin{equation}\label{step_1_of_new_eqn}
    \begin{aligned}
        0 &= \int_{\Omega\intersect\pa^* T_1 \intersect \pa^*T_2}c_{21}\vec{n}_{2,1}\cdot \vec{\nu}_{2,1} \,d\h^{n-1}\\
        &\quad + \sum_{k=3}^{n+1}\left[ \int_{\Omega\intersect\pa^*T_1 \intersect \pa^* T_k}c_{1,n+1}\vec{n}_{1,n+1}\cdot\vec{\nu}_{1,k}\,d\h^{n-1}+ \int_{\Omega\intersect\pa ^*T_2 \intersect \pa^*T_k}c_{2,n+1}\vec{n}_{2,n+1}\cdot \vec{\nu}_{2,k}\,d\h^{n-1} \right]\\
        &\quad +\int_{\pa \Omega \intersect\pa^* T_1 } c_{1,n+1}\vec{n}_{1,n+1}\cdot \vec{\nu}_{\pa\Omega} \,d\h^{n-1}+\int_{\pa \Omega \intersect\pa^* T_2 } c_{2,n+1}\vec{n}_{2,n+1}\cdot \vec{\nu}_{\pa\Omega}\, d\h^{n-1}.
    \end{aligned}
\end{equation}

Likewise, adding the full of \eqref{eqn3.35} below
\begin{equation}\label{eqn3.35}
    \begin{aligned}
         0 &= \int_{T_3} c_{3,n+1}\,{\rm div}(\vec{n}_{3,n+1}) \;dx \\
        &= \sum_{\substack{k=1,\\k\neq 3}}^{n+1}\int_{\Omega\intersect\pa^*T_3 \intersect \pa^* T_k}c_{3,n+1}\vec{n}_{3,n+1}\cdot \vec{\nu}_{3k} \,d\h^{n-1}  +  \int_{\pa \Omega \intersect \pa^*T_3} c_{3,n+1}\vec{n}_{3,n+1}\cdot \vec{\nu}_{\pa\Omega} \,d\h^{n-1}  
    \end{aligned}
\end{equation}
to \eqref{step_1_of_new_eqn}, one gets
\begin{equation}\label{step_1_of_new_eqn_new}
     \begin{aligned}
        0 &= \sum_{1\leq i<j\leq 3}\int_{\Omega\intersect\pa^* T_i \intersect \pa^*T_j}c_{ij}\vec{n}_{i,j}\cdot \vec{\nu}_{i,j} \,d\h^{n-1} \\
        &\qquad + \sum_{k=4}^{n+1}\left[ \int_{\Omega\intersect\pa^*T_1 \intersect \pa^* T_k}c_{1,n+1}\vec{n}_{1,n+1}\cdot\vec{\nu}_{1,k}\,d\h^{n-1}+ \int_{\Omega\intersect\pa ^*T_2 \intersect \pa^*T_k}c_{2,n+1}\vec{n}_{2,n+1}\cdot \vec{\nu}_{2,j}\,d\h^{n-1}\right.\\
        &\qquad + \left.\int_{\Omega\intersect\pa ^*T_3 \intersect \pa^*T_k}c_{3,n+1}\vec{n}_{3,n+1}\cdot \vec{\nu}_{3,j}\,d\h^{n-1}
        \right] + \sum_{j=1}^3\int_{\pa \Omega \intersect\pa^* T_j } c_{j,n+1}\vec{n}_{j,n+1}\cdot \vec{\nu}_{\pa\Omega} \,d\h^{n-1}.
    \end{aligned}
\end{equation}
Following the pattern, by induction, after applying the Gauss-Green theorem and then adding each of the  integrals in the set \[
\left\{ \int_{T_j}c_{j,n+1}\,{\rm div}(\vec{n}_{j,n+1}) \,dx \right\}_{j=4,5,\dots, n}
\]
to \eqref{step_1_of_new_eqn_new} one at a time, one obtains:
\begin{equation}\label{boundary=inside}
    \sum_{1\leq i < j \leq n+1}\int_{\Omega\intersect\pa^* T_i \intersect \pa^* T_j} c_{ij}\vec{n}_{i,j}\cdot \vec{\nu}_{i,j}\, d\h^{n-1} = \sum_{j=1}^{n}\int_{\pa \Omega \intersect\pa^* T_j } c_{j,n+1}\vec{n}_{n+1,j}\cdot \vec{\nu}_{\pa\Omega} \,d\h^{n-1}.
\end{equation}
Upon noting that  $\vec{n}_{i,j}\cdot \vec{\nu}_{i,j} \leq 1$ for all $i,j$ and using \eqref{boundary=inside} above, we have

\begin{equation}\label{1_final}
    \begin{aligned}
        E_0(\mt,\Omega)&= \sum_{1\leq i <j \leq n+1} c_{ij}\h^{n-1} \left( \pa^* T_i \intersect \pa^* T_j \intersect \Omega\right) \\
         &\geq \sum_{1 \leq i < j \leq n+1} \int_{\Omega\intersect\pa^*T_i \intersect \pa^*T_j} c_{ij} \vec{n}_{i,j} \cdot \vec{\nu}_{i,j} \, d\h^{n-1} \\
         &= \sum_{i=1}^n\int_{\pa \Omega \intersect \pa^*T_i}c_{i,n+1}\vec{n}_{n+1,i}\cdot \vec{\nu}_{\pa\Omega} \,d\h^{n-1}.
    \end{aligned}
\end{equation}
Here we also observe that if we replace $\mt$ by $\ms$ in the calculation above, then one instead obtains the identity 
\begin{equation}\label{E_0_s}
E_0(\ms,\Omega) = \sum_{i=1}^n\int_{\pa \Omega \intersect \pa^*S_i}c_{i,n+1}\vec{n}_{n+1,i}\cdot \vec{\nu}_{\pa\Omega} \,d\h^{n-1}. 
\end{equation}
  The goal here is to first argue that $\mt$ and $\ms$ agree on $\pa\Omega$.  To do so, we first decompose the domains of integration in the last line of \eqref{1_final} by writing each $\pa\Omega\cap\pa^*T_i$ as $\pa\Omega\cap\pa S_i$ plus corrections and then use \eqref{E_0_s} to obtain
\begin{equation*}
    \begin{aligned}
        E_0(\mt,\Omega) \geq E_0(\ms,\Omega) &+ \sum_{i=1}^n \int_{\pa \Omega\intersect\pa S_i \intersect \big(\bigcup\{ \pa^* T_j | j=1,2,\dots, n+1, j\neq i\}\big)} c_{i,n+1} \vec{n}_{i,n+1}\cdot \vec{\nu}_{\pa \Omega} \,d\h^{n-1}\\
        & + \sum_{i=1}^n \int_{\pa \Omega\intersect\pa^* T_i \intersect \big(\bigcup\{ \pa^* S_j | j=1,2,\dots, n+1, j\neq i\}\big)} c_{i,n+1} \vec{n}_{n+1,i}\cdot \vec{\nu}_{\pa \Omega}\, d\h^{n-1}.
    \end{aligned}
\end{equation*}
Then, invoking property \eqref{gencij} and using Theorem \ref{lemma1}, upon simplification, we obtain
\begin{equation}\label{11_final_11}
    \begin{aligned}
        E_0(\mt,\Omega) \geq E_0(\ms,\Omega) &+ \sum_{1 \leq i <j \leq n+1}\int_{\pa \Omega\intersect\pa S_i \intersect\pa^*T_j}c_{ij}\vec{n}_{i,j}\cdot \vec{\nu}_{\pa \Omega}\,d\h^{n-1} \\
        &+ \sum_{1 \leq i <j \leq n+1}\int_{\pa \Omega\intersect\pa S_j \intersect\pa^*T_i}c_{ij}\vec{n}_{j,i}\cdot\vec{\nu}_{\pa \Omega}\,d\h^{n-1}.
    \end{aligned}
\end{equation}
Next we observe that each of the integrals in the sum on the right-hand side of \eqref{11_final_11} is non-negative by \eqref{goodnormals} and Theorem \ref{lemma1}. To conclude, from \eqref{11_final_11}, we see that unless
\begin{equation}\label{Peter}
    \sum_{i=1}^{n+1}\h^{n-1} \Big[ (\pa^* T_i \Delta \pa S_i) \intersect \pa \Omega \Big] = 0,
\end{equation}
we would find that $E_0(\mt,\Omega) > E_0(\ms,\Omega)$, contradicting the minimality of $\mt$.
Consequently, the last line of \eqref{1_final} equals $E_0(\ms,\Omega)$,
and so we find that

    \begin{align}
       E_0(\ms,\Omega) \geq E_0(\mt,\Omega) \geq \sum_{1 \leq i < j \leq n+1} \int_{\Omega\intersect\pa^*T_i \intersect \pa^*T_j} c_{ij} \vec{n}_{i,j} \cdot \vec{\nu}_{i,j} \, d\h^{n-1} = E_0(\ms,\Omega),
       \label{innereq}
    \end{align}
again contradicting the minimality of $\mt$ unless the inequalities in \eqref{innereq} are in fact equalities. We now conclude that for all distinct $i, j \in \{1,2,\dots,n+1\}$, 
\begin{equation}\label{lalala}
\vec{n}_{i,j} = \vec{\nu}_{i,j}, \; \h^{n-1} \text{ a.e. on }\Omega\intersect\pa^* T_i \intersect \pa^* T_j .
\end{equation}
Finally, with an appeal to Lemma \ref{hyperplane-theorem}, we conclude that all the sets $\Omega\intersect\pa^* T_i \intersect \pa^* T_j$ are all planar. Together with \eqref{lalala} , one sees that $\pa^* T_i \intersect \pa^* T_j$ agrees with $\pa S_i \intersect \pa S_j$ in $\Omega$. Hence, $\mt=\ms$ and the proof is complete. 
\end{proof}

{\bf Acknowledgements:}
The author would like to thank his advisor Peter Sternberg for the discussions and the valuable feedback on this article. The author would also like to thank Camillo De Lellis for his interest in this project. Finally, the author would like to thank Ramyak Bilas for suggesting the concept of an $n$-simplex and for the discussions on the description of the domain in $\R^n$.

{\bf Conflict of Interest Statement:} The author has no conflicts of interest to report.

{\bf Data Availability Statement:} There is no external data associated with this research.
\bibliographystyle{alpha}
\bibliography{refer}

\end{document}